\documentclass[a4paper]{amsart}
\usepackage{amssymb}
\usepackage{tikz}
\usepackage{epic,curves,mfpic}
\usepackage{hyperref}

\newtheorem{theo}{Theorem}[section]
\newtheorem{lemma}[theo]{Lemma}

\newtheorem{prop}[theo]{Proposition}
\newtheorem{conjec}[theo]{Conjecture}

\theoremstyle{definition}
\newtheorem{defi}[theo]{Definition}
\newcommand{\Z}{\mathbb{Z}}

\newcommand{\B}{\mathcal{B}}
\newcommand{\co}{\mathcal{O}}

\theoremstyle{remark}

\begin{document}

\title{Structure Trees, Networks and Almost Invariant Sets }
\author{M.J. Dunwoody  }


\begin{abstract}
A self-contained account of the theory of structure trees for edge cuts in networks is given.    Applications include a generalisation of the 
Max-Flow Min-Cut Theorem to infinite networks and
a short proof of a conjecture of Kropholler.   This gives a relative version of Stallings' Theorem
on the structure of groups with more than one end.   A generalisation of the Almost Stability Theorem is also obtained, which provides  information
about the structure of the Sageev cubing.
\end{abstract}
\maketitle

\newcommand{\F}{\mathcal{F}}

\newcommand{\N}{\mathbb{N}}
\newcommand{\Q}{\mathbb{Q}}
\newcommand{\cO}{\mathcal{O}}

\newcommand{\D}{\mathcal{D}}
\newcommand{\C}{\mathcal{C}}
\newcommand{\ce}{\mathcal{E}}
\newcommand{\cR}{\mathcal{R}}

\newcommand{\cb}{\mathcal{B}}

\newcommand{\bG}{\bar G}
\newcommand{\bH}{\bar H}
\renewcommand{\d}{{\rm d}}

\newcounter{fig}
\setcounter{fig}{0}

\section{Introduction}
Let $X$ be a connected graph.     A subset $A$ of the vertex set $VX$ is defined to be a {\it cut} if 
$\delta A $ is finite.    Here $\delta A$ is the set of edges with one vertex in $A$ and one vertex in $A^* = VX - A$.
A ray $R$ in $X$ is an infinite sequence $x_1, x_2, \dots $ of distinct vertices such that $x_i, x_{i+1}$ are adjacent for every $i$.
If $A$ is an edge cut, and $R$ is a ray, then there exists an integer $N$ such that for $n > N$ either $x_n \in A$ or $x_n \in A^*$.
We say that $A$ separates rays $R = (x_n), R' = (x_n')$ if for $n$ large enough either $x_n \in A, x_n' \in A^*$ or $x_n \in A^*, x_n' \in A$.
We define $R\sim R'$ if they are not separated by any edge cut.     It is easy to show that $\sim $ is an equivalence relation on the set $\Phi X$ of rays in $X$.
The set $\Omega X = \Phi X/ \sim $ is the set of {edge} ends of $X$.   An edge cut $A$ separates ends $\omega , \omega '$
if it separates rays representing $\omega , \omega '$.     A cut $A$ separates an end $\omega $ and a vertex $v  \in VX$ if for any ray representing
$\omega $,   $R$ is eventually in $A$ and $v \in A^*$ or vice versa.

The number $e(G)$ of a finitely generated group is the number of ends of a Cayley graph of $X$ with respect to a finite generating set $S$.
It turns out that $e(G)$ does not depend on which generating set $S$ is chosen, and that it is always one of $0, 1, 2$ or the cardinal number $c$.
If a finitely generated group $G$ has more than one end, then there is a cut $A \subset G$ (the vertex set of any Cayley graph), which separates 
two rays.     Thus both $A$ and $A^*$ are infinite.      The fact that $\delta A$ is finite is equivalent to the fact that the symmetric difference
$A + As$ is finite for each  $s \in S$, and it is not hard to see that this is equivalent to requiring that $A + Ag$ is finite for every $g \in G$.
A set $A$ with these properties is called a {\it proper almost invariant set}.   Thus a  subset $A$ of $G$ is said to be  {\it almost invariant} if the symmetric difference $A + Ag$ is finite for every
$g \in G$.   In addition $A$ is said to be {\it proper} if both $A$ and $A^* = G -A$ are infinite.   Clearly the  finitely generated   group $G$ has  more than one end
if and only if it has a proper almost invariant subset.    This provides a way of extending our definition to arbitrary groups.
We say that a group $G$ has more than one end it it has a proper almost invariant set.

\begin{theo}\label{SC} 
A group $G$ contains a proper almost invariant subset (i.e. it has more than one end) if and only if it has a non-trivial action
on a tree with finite edge stabilizers.
\end{theo}
This result was proved by Stallings \cite {[St]} for finitely generated groups and was generalized to all groups by Dicks and Dunwoody \cite {[DD]}.
The action of a group $G$ on a tree is {\it trivial} if there is a vertex that is fixed by all of $G$.  Every group has a trivial action
on a tree.

Let $T$ be a tree with directed edge set $ET$.      If $e$ is a directed edge, then let $\bar e$ denote $e$ with the reverse orientation.
If $e, f$ are distinct directed edges then  write $e >f$  if the smallest subtree of $T$ containing $e$ and $f$ is as below.

\begin{figure}[htbp]
\centering
\begin{tikzpicture}[scale=.5]
\draw (-.5, 2)-- (1,2)--(2,1)--(3,2)--(4,1)--(5,2)--(6,1)--(7,2)--(8,1)--(9,2)--(10.5,2)  ;
\draw (.5,2) node {$>$} ;
\draw (9.5,2) node {$>$} ;
\draw (.5,2.1) node [above] {$e$} ;
\draw (9.5,2.1) node [above] {$f$} ;
\end{tikzpicture}
\end{figure}

Suppose the group $G$ acts on $T$.   We say that $g$ {\it shifts } $e$ if either $e >ge$ or $ge > e$.
If for some $e \in ET$ and some $g \in G$,  $g$ shifts $e$,   then $G$ acts non-trivially on a tree $T_e$ obtained by contracting
all edges of $T$ not in the orbit of $e$ or $\bar e$.   In this action there is just one orbit of edge pairs.
Bass-Serre theory tells us that either $G = G_u*_{G_e} G_v$ where $u, v$ are the vertices of $e$ and they are in different orbits  in the contracted tree $T_e$, or $G$ is the 
HNN-group $G = G_u *_{G_e}$ if $u,v$ are in the same $G$-orbit.   
If either case occurs we say that $G$ splits over $G_e$.

If there is no edge $e$ that is shifted by any $g \in G$, (and $G$ acts without involutions, i.e. there is no $g \in G$ such that $ge = \bar e$)
then $G$ must fix a vertex or an end of $T$.   If the action is non-trivial, it fixes an end of $T$, i.e.
$G$ is a union of an ascending sequence of vertex stabilizers,  $G = \bigcup  G_{v_n}$, where $v_1, v_2, \dots   $ is a sequence of 
adjacent vertices and $G_{v_1} \leq G_{v_2}\leq \dots $ and $G \not= G_{v_n} $ for any $n$.

Thus Theorem ~\ref{SC} could be restated as 

\begin{theo} [\cite {[St]}, \cite{[DD]}] 
A group $G$ contains a proper almost invariant subset (i.e. it has more than one end) if and only if it splits over a finite subgroup or
it is countably infinite and locally finite.
\end{theo}
If a group splits over a finite subgroup, then it is possible to choose a generating set $S$ so that the Cayley graph has more than one end.
However for a countably infinite locally finite group there is no Cayley graph with more than one end.

The if part of the theorem is fairly easy to prove.  We now prove  a stronger version of the if part, following \cite {[D]}.

Let $H$ be a subgroup of $G$.   A subset $A$ is  $H$-{\it finite} if $A$ is contained in finitely many right $H$-cosets, i.e. for some finite
set $F$,   $A \subseteq HF$. 
A subgroup $K$ is $H$-finite if and only if $H\cap K$ has finite index in $K$.
   Let $T$ be a $G$-tree and suppose there is an edge $e$ and  vertex $v$.

We say that $e$ {\it  points at} $v$ if there is a subtree of $T$ as below.
We write $e \rightarrow v$.

\begin{figure}[htbp]
\centering
\begin{tikzpicture}[scale=.5]
\draw (-.5, 2)-- (1,2)--(2,1)--(3,2)--(4,1)--(5,2)--(6,1)--(7,2)--(8,1)--(9,2)--(10.5,2)  ;
\draw (.5,2) node {$>$} ;
\draw (.5,2.1) node [above] {$e$} ;
\draw (10.5,2.1) node [above] {$v$} ;
\draw (10.5,2) node {$\bullet$} ;

\end{tikzpicture}
\end{figure}

Let $G[e, v] = \{g \in G | e \rightarrow gv\}$.

If $h \in G$, then $G[e, v]h = G[e,h^{-1}v]$,  since if $e\rightarrow gv,  e \rightarrow gh (h^{-1}v)$.

It follows from this that If $K = G_v$, then $G[e,v]K =G[e,v]$.
Also if $H = G_e$, then $HG[e,v] = G[e,v]$.

If $v = \iota e $, then $G_e = H \leq K = G_v$ and if $A = G[e, \iota e]$, then $A= HAK$.

\begin{figure}[htbp]
\centering
\begin{tikzpicture}[scale=.5]
\draw (-2, 2)-- (2,2) ;
\draw (.0,2) node {$>$} ;
\draw (2,2) node {$\bullet$} ;
\draw (-2,2) node {$\bullet$} ;

\draw (.5,2.1) node [above] {$e$} ;
\draw (-2,2.1) node [above] {$v$} ;
\end{tikzpicture}
\end{figure}

Consider the set $Ax,  x\in G$.  If $g \in A, gx \notin A$ , then $e \rightarrow gv,  \bar e \rightarrow gx v$.
This means that $e$ is on the directed  path joining $gxv$ and $gv$.  This happens if and only
if $g^{-1}e$ is on the path joining $xv$ and $v$.  There are  only finitely many directed  edges in the $G$-orbit
of $e$ in this path.  Hence $g^{-1} \in FH$, where $F$ is finite, and $H = G_e$, and $g \in HF^{-1}$.
Thus $A -Ax^{-1} = HF^{-1}$, i.e. $A- Ax^{-1}$ is $H$-finite.  It follows that both $Ax - A$ and $A-Ax$ are $H$-finite
and so $A + Ax$ is $H$-finite for every $x \in G$,  i.e. $A$ is an  $H$-{\it almost invariant set}.

  If the action on $T$ is non-trivial, then neither $A$ nor $A^*$ is $H$-finite.   We say that $A$ is {\it proper}.
  
  Peter Kropholler has conjectured that the following generalization of Theorem ~\ref{SC} is true for finitely generated groups.
    \begin{conjec}\label{KC}   Let $G$ be a  group and let $H$ be a subgroup.   If there is a proper $H$-almost invariant subset $A$
  such that $A = AH$, then $G$ has a non-trivial action on a tree in which $H$ fixes a vertex $v$ and  every edge  incident with $v$ has an $H$-finite stabilizer.
  
  \end{conjec}

We have seen that  the conjecture   is true if $H$  has one element.    The conjecture has been
proved  for $H$  and $G$ satisfying extra conditions by 
Kropholler \cite {[K90]},  Dunwoody and Roller \cite {[DR]} ,  Niblo \cite {[N]} and  Kar and Niblo \cite {[NK]}.

 If $G$ is the triangle
group $G = \langle a, b | a^2 = b^3 = (ab)^7 =1\rangle$, then $G$ has an infinite cyclic subgroup $H$ for which there 
is a proper $H$-almost invariant set.   Note that in this case $G$ has no non-trivial action on a tree, so the condition $A = AH$ is
necessary in  Conjecture~\ref{KC}.



\begin{center}
    \includegraphics[width=10cm]{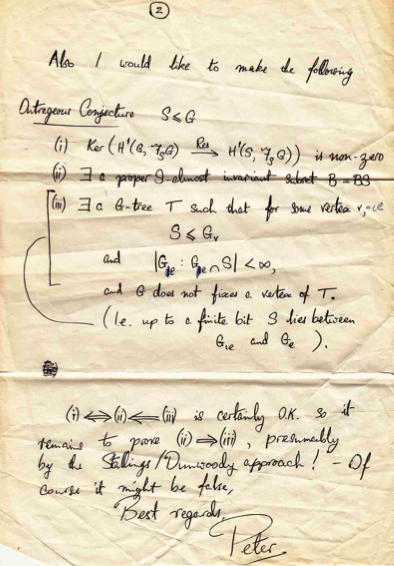}
\end{center}


A discussion of the Kropholler Conjecture is given in \cite{NS}.  I first learned of this conjecture in a letter Peter wrote to me in January 1988,  a page of which is shown here.

We give a proof of the conjecture when $G$ is finitely generated over $H$, i.e. it is generated by $H$ together with a finite subset.

I am very grateful to Peter Kropholler for enjoyable discussions and a very helpful email correpondence about his conjecture.




\renewcommand{\d}{{\rm d}}



The main tool in proving Conjecture ~\ref {KC}  is the theory of structure trees in connected graphs which was initiated in \cite {[D2]} and \cite {[DD]}.
In the next section  a fairly self contained account of this theory is given.     In fact the theory is extended to apply to networks and 
it is shown that the sequence of structure trees obtained for a network is  uniquely determined.
It is this property that is crucial in proving the Kropholler Conjecture.

Another very interesting aspect of extending the theory to networks is that one can obtain non-trivial result for a finite network.
Results for finite networks such as the Max-Flow Min-Cut Theorem (MFMC) and the existence of a Gomory-Hu tree are shown to be special
cases of our results for more general networks.   It is also the case that Stallings' Theorem on the structure of groups with more than one end also
follows from the theory developed here.     It is very pleasing (to me at least) that there is a theory that includes both the Stallings' Theorem and the MFMC.

If $A$ is an almost invariant set,  and $M  = \{ B | B =_a A\}  $ so that  for  $B, C   \in M ,     B +C  = F$ where $F$ is finite.
then  the Almost Stability Theorem of \cite {[DD]}
 shows  that $M$ is the vertex set of a $G$-tree $T$.

We can  define a metric on $M$.
For $B, C \in M$ define $d(B, C) =|B+C|$, and this is a a geodesic metric on $T$.  In the final section a generalisation of this result is proved.
If $H$ is a subgroup of $G$, and $A = HAH$ is an $H$-almost invariant subset,  we now put $M = \{ B | A + B=  HF \} $ where $F$ is finite,  then 
$M$ is a right $G$-set, i.e. it admits an action of $G$ by right multiplication.
If for $B, C \in M$ we put $d(B, C)$ to be the number of cosets $Hx$ in $B+C$, then we have a metric on $M$ as before.    It is
again shown that this is a metric on a tree, thus giving the $H$-almost stability theorem.

At the risk of appearing self-indulgent, I record the history  of the theory of structure trees.

As noted above,
in his breakthrough work  \cite {[St]}  on groups with more than one end, in the late sixties,  Stallings showed that a finitely generated group has a Cayley graph (corresponding to a finite generating set) with more than one end if
and only if it has a certain structure.   At about that time Bass and Serre (see \cite {[DD]} or \cite {[S]})  developed their theory of groups acting on trees and it was clear that
the structure of  a  group with more than one end, as  in Stallings'  Theorem,  was associated with an action on a tree.   In \cite {[D1]} I gave a proof of Stallings' result by constructing
a tree on which the relevant group acted.   This involved showing that if the finitely generated  group $G$  had more than one end, then there is a subset $B \subset G$ such that both $B$ and $B^*$ are infinite,  $\delta B$ is finite, and the set $\ce = \{ gB | g \in G\} $ is almost nested.   
Here we define a cut in a graph $X$ to be a subset $A$ of  $VX$ such that $\delta A$ if finite,  where $\delta A$ is the set of edges with one vertex in 
$A$ and one vertex in $A^* = VX - A$.  The set of all cuts is denoted $\B X$.

A set $\ce $ of cuts is almost nested if for every $A, B \in \ce$ at least one corner of $A$ and $B$ is finite.   A corner of $A, B$ is one of the four sets 
$A\cap B, A^*\cap B, A\cap B^*, A^*\cap B^*$.

In \cite {[D2]} I gave a stronger result by showing that if a group $G$ acts on a graph $X$ with more than one end, then there exists a subset 
$B \in \B X$ such that $B$ and $B^*$ are both infinite and for any $g \in G$ the sets $B$ and $g B$ are nested, i.e. at least one of the four corners
is empty.    The set of all such $gB$ can 
be shown to be the edge set of a tree, called a structure tree.

This result was further extended by Warren Dicks and myself   \cite {[DD]}.  In Chapter II of that book it is shown  that for any graph $X$ the Boolean ring $\B X$
has a particular nested set of generators invariant under the automorphism group of $G$.   
    At the time I thought that the result when applied to  finite graphs was of little interest.
This was partly because an action of a group on a finite tree is always trivial, i.e. there is always a vertex of the tree fixed by the whole group.
This is not the case for groups acting on infinite trees: the theory of such actions is the subject matter of Bass-Serre theory.
Also for a finite graph $X$, there is always a nested set of generators for $\B X$ consisting of single elements subsets.
The belated realisation that the theory developed in \cite {[DD]} might be of some significance for finite networks occurred only recently.

In 2007 Bernhard Kr\" on asked me if one could develop a theory of structure trees for graphs with more that one vertex end rather than
more than one edge end.    These are connected graphs that have more than one infinite component after removing finitely many vertices.
We were able to develop such a theory in \cite {[DK]}.  In the course of  our work on this,  we realised that we could develop a theory of  structure trees 
for finite graphs that generalised the theory of Tutte \cite {[T]}, who obtained a structure tree result for $2$-connected finite graphs that are not $3$-connected.     The theory for vertex cuts is  more complicated than that for edge cuts.
In 2008 I learned about the cactus theorem for min-cuts from Panos Papasoglu.   This theory, due to Dinits, Karsanov and Lomonosov   \cite {[DKL]} (see also \cite {[FF]}) is for finite networks.
It is possible, with a bit more work, to deduce the cactus theorem  from the proof of  Theorem \ref {maintheorem} .
Evangelidou and Papasoglu  \cite {[EP]} have obtained a cactus theorem for edge cuts in infinite graphs, giving a new proof of Stallings' Theorem.
In \cite {[DW]} Diekert and Weiss gave a definition for thin cuts, which is equivalent to the one given in \cite {[DD]},
but which made more apparent the connection with the Max-Flow Min-Cut Theorem.  I also had a very helpful email exchange with Armin Weiss.
Weiss told me about Gomory-Hu trees that are structure trees in finite networks.  

Thinking about these matters finally led me  to think about structure trees for edge cuts in finite graphs and networks and the realisation
that the theory developed in \cite {[DD]} might be of some interest  when applied  to finite networks.

In Section 2 the theory for finite networks is recalled.  The theory is then
generalised to arbitrary networks.    
For any network $N$ we obtain a canonically  determined sequence of trees $T_n$ that provide complete information
about the separation of a pair $s,t$ where each of $s$ and $t$ is either a vertex or an end of $X$.   
It is  possible to obtain all such information from a single tree $T_n$ if $X$ is {\it accessible}.   A graph is accessible if there is an integer
$n$ such that any two ends can be separated by removing at most $n$ edges.  This definition is due to Thomassen and Woess \cite {thomassen1993}.    Other ways of defining accessibility of graphs are discussed.
   There are locally finite vertex transitive
graphs that are inaccessible.  Such graphs are constructed in \cite {Dun} or \cite {Dunwoody1993}.   

The situation for edge cuts contrasts with the situation for vertex cuts.   Thus there is a canonically determined sequence of trees that
separates a pair $s,t$ from the set of vertices or ends of the graph $X$.   For vertex cuts, one can only find a canonically defined structure tree
that separates a pair $\kappa $-inseparable sets or a pair of vertex ends,  where $\kappa $ is the smallest integer for which it is possible to separate
such a pair.

Structure tree theory has been used by several authors to classify infinite graphs that have more than one end and which satisfy different 
transitivity condition.
 For example Macpherson \cite {M} used a structure tree  to classify infinite locally finite distance transitive graphs, and M\" oller \cite {Mo} used these methods
 to classify infinite ended locally finite graphs for which the automorphism group acted transitively on the ends.
In \cite {thomassen1993} Thomassen and Woess obtain a number of results using structure trees.   They show for example  that if $r$ is prime, then 
 a connected, $r$-regular, $1$-transitive graph with  more than one end, is a tree.
  
 This paper  incorporates two papers \cite {[D4]} and \cite {[D5]} that have appeared on arXiv. 
I am very grateful to Peter Kropholler and Armando Martino who made a careful study of the  earlier papers.   I have included
their suggestions and corrections in this version.
\section {Networks and Structure Trees}
\subsection {Finite Networks}
In this subsection we define our terminology, but restrict attention  to networks based on finite graphs.  We recall the Max-Flow Min-Cut Theorem
and state the result that our more general theory gives for finite networks.  We illustrate the theory with examples.

We define a network  $N$ to be a finite simple, connected graph  $X$  and a map $c : EX \rightarrow \{1, 2, \dots \}$.

Let $s, t \in VX$.   An {\it $(s, t)$-flow } in $N$ is a map  $f :  EX \rightarrow \{ 0, 1, 2, \dots \}$  together with an assignment of a direction to each edge  $e$ so that its vertices are $\iota e$ and $\tau e$  and the following holds.

\begin {itemize}
\item [(i)]   For each $e \in EX$,  $f(e) \leq c(e)$.

\item [(ii)] If we put $f^+ (v) =  \Sigma  ( f(e) | \iota e = v ) $ and $f^-(v) = \Sigma (f (e) | \tau e = v )$, then for every
$v \in VX, v \not= s, v \not= t$,  we have $f^+(v) = f^-(v)$. 
That is, at every vertex except $s$ or $ t$, the flow into that vertex is the same as the flow out.

\end {itemize}  

\begin{figure}[htbp]
\centering
\begin{tikzpicture}[scale=.8]
\draw [-> , very thick] (0,0) -- ( 3,0) ;
\draw [very thick] (3,0) --(6,0) ;
\filldraw (0,0) circle (2pt);
\filldraw (6,0) circle (2pt);
\draw (0,.2) node [above] {$\iota e$};
\draw (3,.2) node [above] {$e$};
\draw (6,.2) node [above] {$\tau e$};

\end{tikzpicture}
\end{figure}

It is easy to show that in an $(s,t)$-flow,    $f^+(s) - f^-(s)   =  -(f^+(t) - f^-(t))$.   The {\it value} of the flow is defined to be $|f| = |f^+(s) - f^-(s)|$.
We define a {\it cut} in $X$ to be a subset  $A$  of $VX$,  $A\not= \emptyset, A\not= VX$.   If $A$ is a cut then so is its complement $A^*$.
If $N$ is a network and $A \subset VX$ is a cut, then the capacity $c(A)$ of $A$ is the sum $c(A) = \Sigma \{ c(e) | e = (u,v), u \in A. v \in A^*\}$.
We define  $\delta A$ to be the set of edges with one vertex in $A$ and one in $A^*$, so that $c(A)$ is the sum of the values $c(e)$ as $e$ ranges over the edges of $\delta A$.    We could replace each edge $e$  of $X$ with $c(e)$ edges joining the same two vertices and then have a theory
in which the capacity of a cut is the number of edges in $\delta A$.

In Figure \ref{fig:Tutte} a network is shown, together with a max-flow (which has value $7$), together with a corresponding min-cut.

\begin {theo} [The Max-Flow Min-Cut Theorem \cite {[FF1]}]  The maximum value of an  $(s, t)$-flow is the minimal capacity of a cut separating $s$ and $t$.
\end {theo}

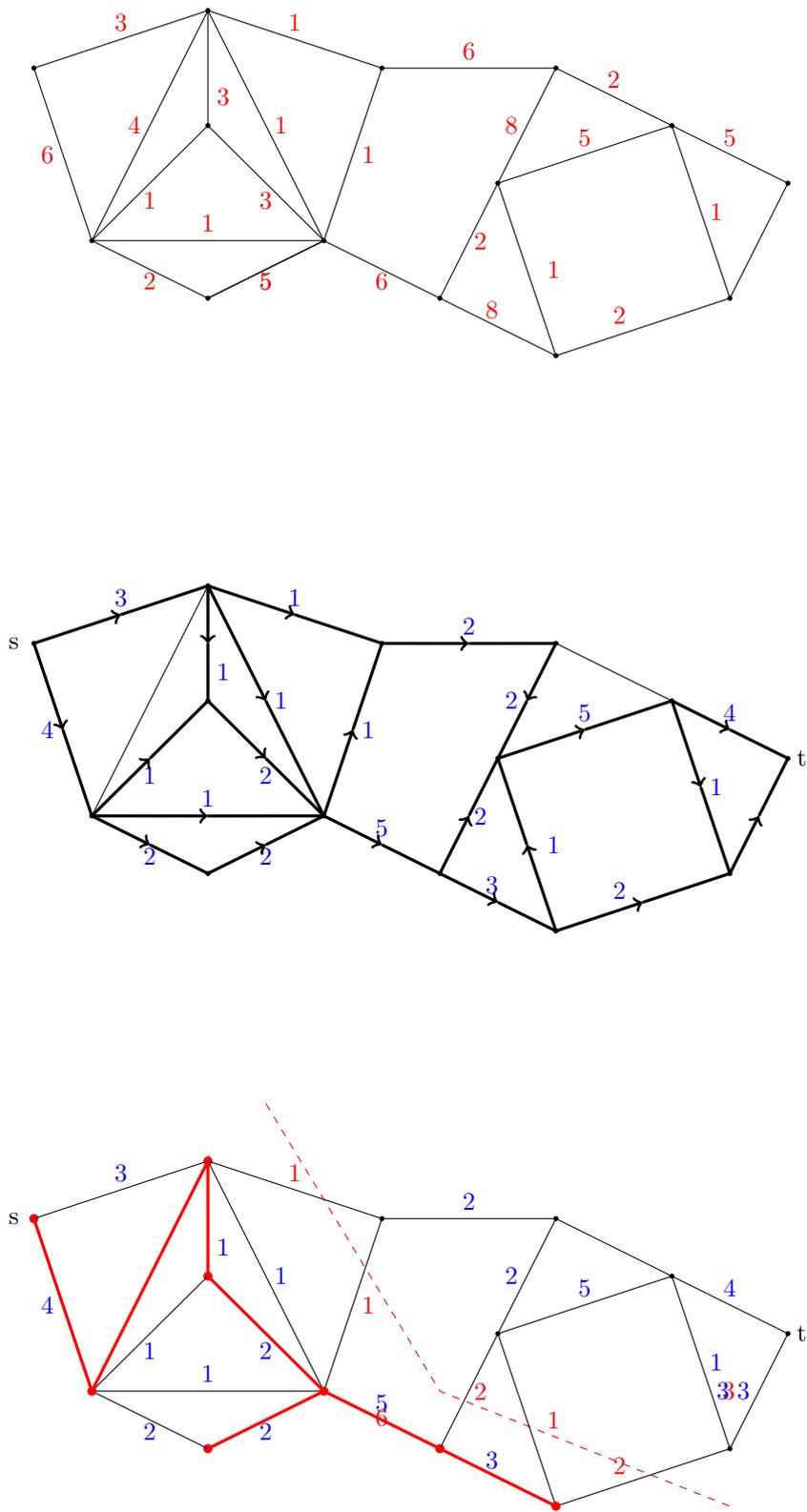
\begin{figure}[htbp]
\centering
\begin{tikzpicture}[scale=.8]


\path (5,28) coordinate (p1); 
\path (6,31) coordinate (p2);  
\path (3,32) coordinate (p3); 
\path (3,30) coordinate (p4);  
\path (0,31) coordinate (p5);  
\path (1,28) coordinate (p6);  
\path (3,27) coordinate (p7);  

\path (9,26) coordinate (p17);  
\path (9,25) coordinate (nn);  
\path (8,29) coordinate (p18);   
\path (11,30) coordinate (p19);  
\path (12,27) coordinate (p20); 
\path (7,27) coordinate (p21); 
\path (9,31) coordinate (p22); 
\path (13,29) coordinate (p23); 
\path (10.8,25.7) coordinate (p24);
\path (11.5,23.5) coordinate (p25);

\draw [red] (10,30.5) node [above] {2};

\draw [red] (12,29.5) node [above] {5};

\draw [red] (8.5,30) node [left] {8};

\draw [red] (12.25,8) node [left] {3};

\draw [red] (9.5, 29.5) node [above] {5};

\draw [red] (11.5, 28.5) node [right] {1};

\draw [red] (7.7,27.7) node [above] {2};

\draw [red] (8.7,27.5) node [right] {1};

\draw [red] (7.9,26.5) node [above] {8};

\draw [red] (10.1,26.4) node [above] {2};

\draw [red] (6,7.25) node [above] {6};

\draw [red] (4,27) node [above] {5};

\draw [red] (2,27) node [above] {2};

\draw [red] (4,27) node [above] {5};

\draw [red] (2,29) node [below] {1};

\draw [red] (4,29) node [below] {3};

\draw [red] (7.5, 31) node [above] {6} ;

\draw [red] (6, 27) node [above] {6} ;

\draw [red] (3,30.5) node [right] {3};

\draw [red] (3,28) node [above] {1};

\draw [red] (5.5, 29.5) node [right] {1};

\draw [red] (4, 30) node [right] {1};

\draw [red] (1.5,31.5) node [above] {3};

\draw [red] (4.5, 31.5) node [above] {1};

\draw [red] (.5, 29.5) node [left] {6};

\draw [red] (2, 30) node [left] {4};

\filldraw (p2) circle (1pt);
\filldraw (p3) circle (1pt);
\filldraw (p4) circle (1pt);

\filldraw (p5) circle (1pt);
\filldraw (p7) circle (1pt);

\filldraw (p18) circle (1pt);
\filldraw (p19)circle (1pt);
\filldraw (p21) circle (1pt);

\filldraw (p22) circle (1pt);

\filldraw (p23) circle (1pt);

\draw (p1)--(p6) ;
\draw (p1) -- (p2) -- (p3)--(p1)--(p4)--(p3) -- (p5)--(p6)--(p3);
\draw (p4)--(p6) ;
\draw  (p1)--(p7)--(p1) ;
\draw (p7) -- (p6) ;

\draw (p17)--(p18)--(p19)--(p20)--(p17);
\draw (p1)--(p17) ;
\draw (p20)--(p23)--(p22)--(p22)--(p21);
\draw (p2) -- (p22) ;

\filldraw (p1) circle (1pt);
\filldraw (p6) circle (1pt);
\filldraw (p17)  circle (1pt);
\filldraw (p20) circle (1pt);

\path (5,18) coordinate (p1); 
\path (6,21) coordinate (p2);  
\path (3,22) coordinate (p3); 
\path (3,20) coordinate (p4);  
\path (0,21) coordinate (p5);  
\path (1,18) coordinate (p6);  
\path (3,17) coordinate (p7);  

\path (9,16) coordinate (p17);  
\path (9,15) coordinate (nn);  
\path (8,19) coordinate (p18);   
\path (11,20) coordinate (p19);  
\path (12,17) coordinate (p20); 
\path (7,17) coordinate (p21); 
\path (9,21) coordinate (p22); 
\path (13,19) coordinate (p23); 
\path (10.8,15.7) coordinate (p24);
\path (11.5,13.5) coordinate (p25);

\draw (p1) -- (p2) -- (p3)--(p1)--(p4)--(p3) -- (p5)--(p6)--(p3);

\draw [blue] (12,19.5) node [above] {4};

\draw [blue] (8.5,20) node [left] {2};

\draw [blue] (12.15,8) node [left] {3};

\draw [blue] (9.5, 19.5) node [above] {5};

\draw [blue] (11.5, 18.5) node [right] {1};

\draw [blue] (7.7,17.7) node [above] {2};

\draw [blue] (8.7,17.5) node [right] {1};

\draw [blue] (7.9,16.5) node [above] {3};

\draw [blue] (10.1,16.4) node [above] {2};

\draw [blue] (6,17.5) node [above] {5};

\draw [blue] (4,17) node [above] {2};

\draw [blue] (2,17) node [above] {2};

\draw [blue] (2,19) node [below] {1};

\draw [blue] (4,19) node [below] {2};

\draw [blue] (7.5, 21) node [above] {2} ;

\draw [blue] (3,20.5) node [right] {1};

\draw [blue] (3,18) node [above] {1};

\draw [blue] (5.5, 19.5) node [right] {1};

\draw [blue] (4, 20) node [right] {1};

\draw [blue] (1.5,21.5) node [above] {3};

\draw [blue] (4.5, 21.5) node [above] {1};

\draw [blue] (.5, 19.5) node [left] {4};


\filldraw (p2) circle (1pt);
\filldraw (p3) circle (1pt);
\filldraw (p4) circle (1pt);

\filldraw (p5) circle (1pt);
\filldraw (p7) circle (1pt);

\filldraw (p18) circle (1pt);
\filldraw (p19)circle (1pt);
\filldraw (p21) circle (1pt);

\filldraw (p22) circle (1pt);

\filldraw (p23) circle (1pt);

\draw [->, very thick] (p6)-- (3,18) ;
\draw [ very thick] (3,18) -- (p1) ;
\draw [->, very thick] (p1) --(5.5,19.5) ;
\draw [very thick] (5.5,19.5) --  (p2)  ;
\draw [very thick] (4.5,21.5) -- (p2) ;
\draw [->, very thick] (p3) --(4.5, 21.5) ;
\draw[very thick]  (4,19)--(p1) ;
\draw [->, very thick] (p4)--(4,19);
\draw [->, very thick]  (p5) -- (1.5, 21.5) ;
\draw [very thick] (1.5, 21.5)-- (p3) ;
\draw [->, very thick] (p5) -- (.5, 19.5) ;
\draw [->, very thick] (p3) --( 3, 21) ;
\draw [very thick] (3,21) --(p4)--(2, 19)  ;
\draw [very thick] (.5, 19.5) -- (p6) ;
\draw [very thick]  (p1)--(4,20);
\draw [->, very thick]  (p3)--(4,20);

\draw [->, very thick] (p6)--(2,19) ;
\draw  [->, very thick] (p7)--(4, 17.5) ;
\draw [very thick] (4, 17.5)--(p1) ;

\draw [->, very thick] (p6) --(2, 17.5) ;
\draw [very thick] (2, 17.5) -- (p7) ;

\draw [->, very thick]  (p17)--(8.5, 17.5) ;
\draw [very thick] (8.5, 17.5)--(p18) ;

\draw [->, very thick] (p18) --(9.5,19.5)  ;
\draw [very thick] (9.5, 19.5)--(p19) ;

\draw [very thick] (p20)--(10.5, 16.5);
\draw [->, very thick] (p17)--(10.5, 16.5) ;

\draw [->, very thick] (p1)--(6,17.5) ;
\draw [very thick] (6, 17.5)--(7,17) ;

\draw [->, very thick] (7,17)--(8,16.5) ;

\draw [->, very thick] (7,17)--(7.5,18) ;
\draw [very thick] (7.5,18)--(p18) ;
\draw [very thick] (8,16.5)--(p17) ;

\draw [->, very thick] (p22)--(8.5,20) ;
\draw [very thick] (8.5, 20)--(p18) ;

\draw [ ->, very thick] (p20)--(12.5, 18) ;
\draw [very thick]  (12.5, 18) --(p23) ;
\draw [->, very thick]  (p19)-- (12, 19.5) ;
\draw [very thick] (12,19.5)--(p23) ;

\draw [->, very thick] (p19)--(11.5,18.5) ;
\draw [very thick] (11.5, 18.5)--(p20) ;

\draw [->, very thick] (p2) --(7.5, 21) ;
\draw [very thick ] (7.5,21)-- (p22) ;
\draw (p22) --(p19)  ;

\filldraw (p1) circle (1pt);
\filldraw (p6) circle (1pt);
\filldraw (p17)  circle (1pt);
\filldraw (p20) circle (1pt);

\draw (p5) node [left] {{}\hskip-5mm s};

\draw (p23) node [right] {t};



\path (5,8) coordinate (p1); 
\path (6,11) coordinate (p2);  
\path (3,12) coordinate (p3); 
\path (3,10) coordinate (p4);  
\path (0,11) coordinate (p5);  
\path (1,8) coordinate (p6);  
\path (3,7) coordinate (p7);  

\path (9,6) coordinate (p17);  
\path (9,5) coordinate (nn);  
\path (8,9) coordinate (p18);   
\path (11,10) coordinate (p19);  
\path (12,7) coordinate (p20); 
\path (7,7) coordinate (p21); 
\path (9,11) coordinate (p22); 
\path (13,9) coordinate (p23); 
\path (10.8,5.7) coordinate (p24);
\path (11.5,3.5) coordinate (p25);

\draw [dashed, red]  (4,13) --(7,8)--(12, 6);

\draw [blue] (12,9.5) node [above] {4};

\draw [blue] (8.5,10) node [left] {2};

\draw [blue] (12.5,8) node [left] {3};

\draw [blue] (9.5, 9.5) node [above] {5};

\draw [blue] (11.5, 8.5) node [right] {1};

\draw [red] (7.7,7.7) node [above] {2};

\draw [red] (8.7,7.5) node [right] {1};

\draw [blue] (7.9,6.5) node [above] {3};

\draw [red] (10.1,6.4) node [above] {2};

\draw [blue] (6,7.5) node [above] {5};

\draw [blue] (4,7) node [above] {2};

\draw [blue] (2,7) node [above] {2};

\draw [blue] (2,9) node [below] {1};

\draw [blue] (4,9) node [below] {2};

\draw [blue] (7.5, 11) node [above] {2} ;

\draw [blue] (3,10.5) node [right] {1};

\draw [blue] (3,8) node [above] {1};

\draw [red] (5.5, 9.5) node [right] {1};

\draw [blue] (4, 10) node [right] {1};

\draw [blue] (1.5,11.5) node [above] {3};

\draw [red] (4.5, 11.5) node [above] {1};

\draw [blue] (.5, 9.5) node [left] {4};


\filldraw (p2) circle (1pt);
\filldraw [red] (p3) circle (2pt);
\filldraw [red] (p4) circle (2pt);

\filldraw [red]  (p5) circle (2pt);
\filldraw [red] (p7) circle (2pt);

\filldraw (p18) circle (1pt);
\filldraw (p19)circle (1pt);
\filldraw [red] (p21) circle (2pt);

\filldraw (p22) circle (1pt);

\filldraw (p23) circle (1pt);

\draw (p1)--(p6) ;
\draw (p1) -- (p2) -- (p3) -- (p1) ;
\draw (p1)--(p4)--(p3) -- (p5)--(p6)--(p3);

\draw (p4)--(p6) ;

\draw (p7) -- (p6) ;

\draw (p17)--(p18)--(p19)--(p20)--(p17);
\draw (p1)--(p17) ;
\draw (p20)--(p23)--(p22)--(p22)--(p21);
\draw (p2) -- (p22) ;
\draw [red, very thick]   (p4)--(p1)-- (p21) -- (p17) ;
\draw [red, very thick]  (p4)-- (p3) --(p6) -- (p5)  ;
\draw [red, very thick]  (p1)--(p7)  ;
\filldraw [red] (p1) circle (2pt);
\filldraw [red] (p6) circle (2pt);
\filldraw [red] (p17)  circle (2pt);
\filldraw (p20) circle (1pt);

\draw (p5) node [left] {{}\hskip-5mm s};

\draw (p23) node [right] {t};

\end{tikzpicture}
\vskip-2mm\caption{Max-Flow Min-Cut Theorem}\label{fig:Tutte}\vskip-3mm \end{figure}
In the proof of this result it is shown that one obtains a min-cut from a max-flow as the set of vertices that are connected to $s$ by a path in which each edge has some unused capacity.    Thus in Figure \ref {fig:Tutte}  the  min-cut  vertices are shown in red and the  edges with unused capacity used in the construction of the max-flow are also shown
in red.

In this paper it is shown that for any finite network there is a uniquely determined network based on a {\it structure tree} that provides a convenient way of encoding the minimal flow between
any pair of vertices.  Specifically we will prove the following theorem.

\begin {theo}\label {maintheorem}  Let $N(X)$ be a finite network.    There is a uniquely determined network $N(T)$ based on a tree  $T$ and an injective map $\nu : VX \rightarrow VT$, such
that  the maximum value of an $(s,t)$-flow in $X$ is the maximum value of a $(\nu s, \nu t)$-flow in $N(T)$.   Also,  for any edge $e' \in ET$, there are vertices
$s, t \in VX$ such that $e'$ is on the geodesic joining $\nu s$ and $\nu t$ and $c(e')$ is the capacity of a minimal $(s,t)$-cut.

\end {theo}

An example of a network and its structure tree are shown in Figure \ref {fig:Tree} .      Thus in this network the max-flow between $u$ and $p$ is $12$.
One can read off a corresponding min-cut by removing the corresponding edge from the  structure tree.  Thus a min-cut separating $u$ and $p$ is
 $\{q,  r, s, t, u, v, w  \} $.    The map $\nu $ need not be surjective.    In our example there is a single vertex  $z$  that is not in the image
 of $\nu $ shown in bold.     One can get a structure tree for which $\nu $ is bijective by contracting one of the four edges incident with this vertex.
 The tree then obtained is a Gomory-Hu tree \cite {GH}.
The structure tree constructed in the proof of Theorem \ref {maintheorem} is uniquely determined and is therefore invariant under the automrophism
group of the network.    The tree obtained by contracting one of the four edges is no longer uniquely determined as one gets a different
tree for each of the four choices.
 In some cases this would mean that the structure tree  did not admit
 the automorphism group of the network.   Thus for example if the automorphism group of $X$ is transitive on $VX$ and $c(e) = 1$ for every edge,
 then the structure tree would have $n$ vertices of degree one, where $n = |VX|$ and one vertex of degree $n$.  Clearly this structure tree will admit
 the automorphism group of $X$, but if one edge is contracted to get a tree with $n$ vertices, then the new tree will not admit the automorphism group.
 
 Not every min-cut separating a pair of vertices can be obtained from the structure tree.   The min-cuts obtained are the ones that are optimally nested with the cuts of equal or smaller capacity.    In our example there are four cuts of capacity $12$ corresponding to edges in the structure tree incident with $z$.   However there are other cuts of capacity $12$.    Thus there are two min-cuts in the structure tree separating $k$ and $h$,
 but there are in fact four min-cuts separating $k$ and $h$.    In  \cite {[DKL]}  it is shown that the min-cuts separating two vertices correspond to the edge cuts 
 in a cactus, which is a connected graph in which each edge belongs to at most one cycle.
 The cactus of min-cuts separating $k$ and $h$ is a $4$-cycle.

\begin{figure}[htbp]
\centering
\begin{tikzpicture}[scale=.7]
\
\path (5,8) coordinate (p1); 
\path (6,11) coordinate (p2);  
\path (3,12) coordinate (p3); 
\path (3,10) coordinate (p4);  
\path (0,11) coordinate (p5);  
\path (1,8) coordinate (p6);  
\path (3,7) coordinate (p7);  
\path (3,4) coordinate (p8);  
\path (1,4) coordinate (p9);  
\path (1,2) coordinate (p10);  
\path (3,2) coordinate (p11);  
\path (5,4) coordinate (p12);   
\path (5,2) coordinate (p13);   
\path (5,0) coordinate (p14);   
\path (7,2) coordinate (p15);   
\path (7,4) coordinate (p16);   
\path (9,6) coordinate (p17);  
\path (9,5) coordinate (nn);  
\path (8,9) coordinate (p18);   
\path (11,10) coordinate (p19);  
\path (12,7) coordinate (p20); 
\path (7,7) coordinate (p21); 
\path (9,11) coordinate (p22); 
\path (13,9) coordinate (p23); 
\path (10.8,5.7) coordinate (p24);
\path (11.5,3.5) coordinate (p25);

\draw [red] (10,10.5) node [above] {2};

\draw [red] (12,9.5) node [above] {1};

\draw [red] (8.5,10) node [left] {8};

\draw [red] (12.5,8) node [left] {3};

\draw [red] (9.5, 9.5) node [above] {2};

\draw [red] (11.5, 8.5) node [right] {1};

\draw [red] (7.7,7.7) node [above] {2};

\draw [red] (8.7,7.5) node [right] {1};

\draw [red] (7.9,6.5) node [above] {8};

\draw [red] (10.1,6.4) node [above] {3};

\draw [red] (7.5,4.5 ) node [above] {2};

\draw [red] (8,4) node [right] {1};

\draw [red] (6,1) node [above] {8};

\draw [red] (7,3) node [right] {3};

\draw [red] (7,5 ) node [above] {3};

\draw [red] (6,4) node [above] {1};

\draw [red] (6,2) node [above] {4};

\draw [red] (5,3) node [right] {7};

\draw [red] (5,1) node [right] {1};

\draw [red] (6,7.5) node [above] {6};

\draw [red] (4,3) node [above] {1};

\draw [red] (4,2) node [above] {4};

\draw [red] (4,1) node [above] {7};

\draw [red] (4,7) node [above] {5};

\draw [red] (2,3) node [right] {5};

\draw [red] (2,2) node [above] {1};


\draw [red] (2,7) node [above] {2};

\draw [red] (4,7) node [above] {5};



\draw [red] (2,9) node [below] {1};

\draw [red] (4,9) node [below] {2};

\draw [red] (1, 6) node [left] {6};

\draw [red] (1,3) node [left] {4};

\draw [red] (3,10.5) node [right] {3};

\draw [red] (3,8) node [above] {1};

\draw [red] (5.5, 9.5) node [right] {6};

\draw [red] (4, 10) node [right] {1};

\draw [red] (1.5,11.5) node [above] {3};

\draw [red] (4.5, 11.5) node [above] {1};

\draw [red] (.5, 9.5) node [left] {6};

\draw [red] (2, 10) node [left] {4};

\filldraw (p2) circle (1pt);
\filldraw (p3) circle (1pt);
\filldraw (p4) circle (1pt);
\filldraw (p5) circle (1pt);
\filldraw (p7) circle (1pt);
\filldraw (p9) circle (1pt);
\filldraw (p10) circle (1pt);
\filldraw (p11) circle (1pt);
\filldraw (p12) circle (1pt);
\filldraw (p13) circle (1pt);
\filldraw (p14) circle (1pt);
\filldraw (p15) circle (1pt);
\filldraw (p16) circle (1pt);
\filldraw (p18) circle (1pt);
\filldraw (p19)circle (1pt);
\filldraw (p21) circle (1pt);
\filldraw (p22) circle (1pt);
\filldraw (p23) circle (1pt);

\draw (p1)--(p6) ;
\draw (p1) -- (p2) -- (p3)--(p1)--(p4)--(p3) -- (p5)--(p6)--(p3);
\draw (p4)--(p6) ;
\draw  (p1)--(p7)--(p1) ;
\draw (p7) -- (p6) -- (p9)--(p10)--(p11)--(p9);
\draw (p11)--(p15) ;-
\draw (p12)--(p14)--(p11)--(p12);
\draw (p14)--(p15)--(p16)--(p12)--(p17)--(p16);
\draw (p15)--(p17)--(p18)--(p19)--(p20)--(p17);
\draw (p1)--(p17) ;
\draw (p20)--(p23)--(p22)--(p22)--(p21);

\filldraw (p1) circle (1pt);
\filldraw (p6) circle (1pt);
\filldraw (p17)  circle (1pt);
\filldraw (p20) circle (1pt);

\draw (p1) node [right] {\hskip1.5mm a};
\draw (p2) node [right] {\hskip0.6mm b};
\draw (p3) node [above] {c};
\draw (p4) node [below] { d};
\draw (p7) node [below] {\hskip2.5mm e};
\draw (p5) node [left] {{}\hskip-5mm g};
\draw (p6) node [left] {h};
\draw (p9) node [left] {i};
\draw (p10) node [left] { j};
\draw (p11) node [below] {\hskip-2mm k};
\draw (p12) node [above] {\hskip-2mm l};
\draw (p13) node [above] {\hskip-5mm m};
\draw (p14) node [below] {n};
\draw (p16) node [below] {\hskip-4.4mm o};
\draw (p15) node [below] {\hskip2mm p};
\draw (nn) node [above] {\hskip0.8mm q};
\draw (p18) node [left] {r};
\draw (p19) node [above] {\hskip2mm s};
\draw (p20) node [right] {t};
\draw (p21) node [below] {\hskip-1mm u};
\draw (p22) node [above] {v};
\draw (p23) node [right] {w};

\end{tikzpicture}
\end{figure}

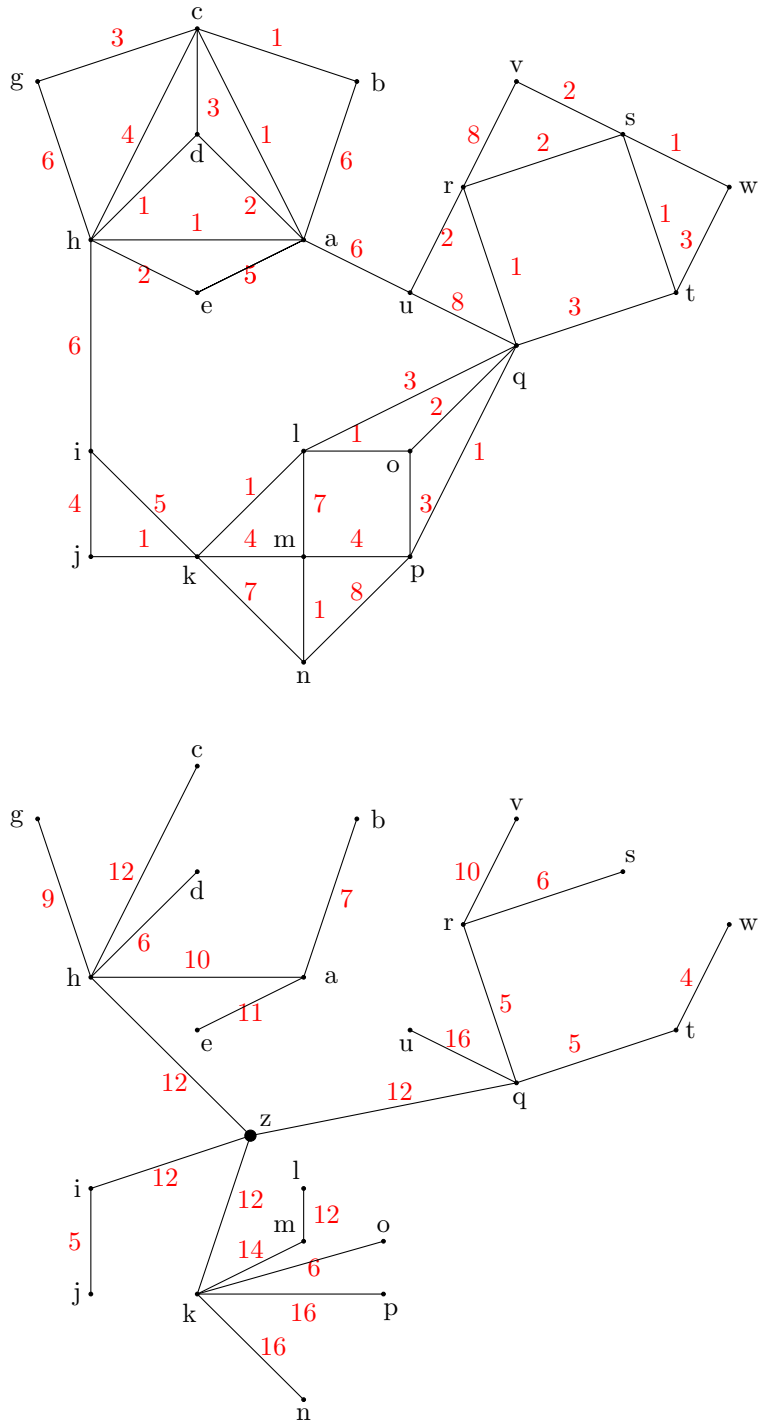
\begin{figure}[htbp]
\centering
\begin{tikzpicture}[scale=.7]
\

\path (5,8) coordinate (p1); 
\path (6,11) coordinate (p2);  
\path (3,12) coordinate (p3); 
\path (3,10) coordinate (p4);  
\path (0,11) coordinate (p5);  
\path (1,8) coordinate (p6);  
\path (3,7) coordinate (p7);  
\path (4,5) coordinate (p8);  
\path (1,4) coordinate (p9);  
\path (1,2) coordinate (p10);  
\path (3,2) coordinate (p11);  
\path (5,4) coordinate (p12);   
\path (5,3) coordinate (p13);   
\path (5,0) coordinate (p14);   
\path (6.5,2) coordinate (p15);   
\path (6.5,3) coordinate (p16);   
\path (9,6) coordinate (p17);  
\path (9,5) coordinate (nn);  
\path (8,9) coordinate (p18);   
\path (11,10) coordinate (p19);  
\path (12,7) coordinate (p20); 
\path (7,7) coordinate (p21); 
\path (9,11) coordinate (p22); 
\path (13,9) coordinate (p23); 
\path (10.8,5.7) coordinate (p24);
\path (11.5,3.5) coordinate (p25);



\draw [red] (8.5,10) node [left] {10};

\draw [red] (12.5,8) node [left] {4};

\draw [red] (9.5, 9.5) node [above] {6};



\draw [red] (8.5,7.5) node [right] {5};

\draw [red] (7.9,6.5) node [above] {16};

\draw [red] (10.1,6.4) node [above] {5};

\draw [red] (6.8,5.5 ) node [above] {12};




\draw [red] (4,3.8 ) node {12};


\draw [red] (5,2) node [below] {16};

\draw [red] (5,3.5) node [right] {12};

\draw [red] (4,1) node [right] {16};


\draw [red] (5.5,2.5) node [left] {6};

\draw [red] (4,2.5) node [above] {14};


\draw [red] (4,7) node [above] {11};

\draw [red] (2.4,4.2) node {12};


\draw [red] (2,9) node [below] {6};


\draw [red] (3, 6) node [left] {12};

\draw [red] (1,3) node [left] {5};


\draw [red] (3,8) node [above] {10};

\draw [red] (5.5, 9.5) node [right] {7};




\draw [red] (.5, 9.5) node [left] {9};

\draw [red] (2, 10) node [left] {12};

\filldraw (p2) circle (1pt);
\filldraw (p3) circle (1pt);
\filldraw (p4) circle (1pt);
\filldraw (p5) circle (1pt);
\filldraw (p7) circle (1pt);
\filldraw (p8) circle (3pt);
\filldraw (p9) circle (1pt);
\filldraw (p10) circle (1pt);
\filldraw (p11) circle (1pt);
\filldraw (p12) circle (1pt);
\filldraw (p13) circle (1pt);
\filldraw (p14) circle (1pt);
\filldraw (p15) circle (1pt);
\filldraw (p16) circle (1pt);
\filldraw (p18) circle (1pt);
\filldraw (p19)circle (1pt);
\filldraw (p21) circle (1pt);
\filldraw (p22) circle (1pt);
\filldraw (p23) circle (1pt);

\draw (p1)--(p6);
\draw (p1) -- (p2) ;
\draw (p5)--(p6)--(p3);
\draw (p4)--(p6) ;
\draw  (p1)--(p7);
\draw (p6)--(p8)--(p9)--(p10)  ;
\draw (p11)--(p13);
\draw (p11)--(p15) ;-
\draw (p12)--(p13);
\draw (p17)--(p8)--(p11) ;
\draw (p14)--(p11) ;
\draw (p17)--(p18)--(p19) ;
\draw (p20)--(p17);
\draw (p21)--(p17) ;
\draw (p20)--(p23) ;
\draw (p22)--(p18) ;
\draw (p16)--(p11);

\filldraw (p1) circle (1pt);
\filldraw (p6) circle (1pt);
\filldraw (p17)  circle (1pt);
\filldraw (p20) circle (1pt);

\draw (p1) node [right] {\hskip1.5mm a};
\draw (p2) node [right] {\hskip0.6mm b};
\draw (p3) node [above] {c};
\draw (p4) node [below] { d};
\draw (p7) node [below] {\hskip2.5mm e};
\draw (p8) node [above] {\hskip4mm z};
\draw (p5) node [left] {{}\hskip-5mm g};
\draw (p6) node [left] {h};
\draw (p9) node [left] {i};
\draw (p10) node [left] { j};
\draw (p11) node [below] {\hskip-2mm k};
\draw (p12) node [above] {\hskip-2mm l};
\draw (p13) node [above] {\hskip-5mm m};
\draw (p14) node [below] {n};
\draw (p16) node [left,above] { o};
\draw (p15) node [below] {\hskip2mm p};
\draw (p17) node [below] {\hskip0.8mm q};
\draw (p18) node [left] {r};
\draw (p19) node [above] {\hskip2mm s};
\draw (p20) node [right] {t};
\draw (p21) node [below] {\hskip-1mm u};
\draw (p22) node [above] {v};
\draw (p23) node [right] {w};

\end{tikzpicture}
\vskip-2mm\caption{Network and structure tree}\label{fig:Tree}\vskip-3mm
\end{figure}
\eject
\subsection{The Algebra of Cuts}

Let $N$ be a network based on the graph $X$.   We now allow $X$ to be infinite.    Thus $N$ is a simple connected graph with
a map $c: EX \rightarrow  \{ 1,2, \dots  \}$.   If $A$ is a cut, i.e. a subset of $VX$ for which $\delta A$ is finite, then let $c(A) = \Sigma \{ c(e) | e \in \delta A\}$.   Note that we do not assume that $X$ is locally finite.   It is convenient from here on to allow $\emptyset $ and $VX$ to be cuts.  Thus the set of
cuts is a Boolean ring $\B X$.

A ray $R$ in $X$ is an infinite sequence $x_1, x_2, \dots $ of distinct vertices such that $x_i, x_{i+1}$ are adjacent for every $i$.
If $A$ is an edge cut, and $R$ is a ray, then there exists an integer $N$ such that for $n > N$ either $x_n \in A$ or $x_n \in A^*$.
We say that $A$ separates rays $R = (x_n), R' = (x_n')$ if for $n$ large enough either $x_n \in A, x_n' \in A^*$ or $x_n \in A^*, x_n' \in A$.
We define $R\sim R'$ if they are not separated by any edge cut.     It is easy to show that $\sim $ is an equivalence relation on the set $\Phi X$ of rays in $X$.
The set $\Omega X = \Phi X/ \sim $ is the set of {edge} ends of $X$.   An edge cut $A$ separates ends $\omega , \omega '$
if it separates rays representing $\omega , \omega '$.     A cut $A$ separates an end $\omega $ and a vertex $v  \in VX$ if for any ray representing
$\omega $,   $R$ is eventually in $A$ and $v \in A^*$ or vice versa.

A   cut  $A$  is  defined to be  {\it thin   with respect to }$u,v \in VX\cup \Omega X$ if it separates some $u,v    \in VX\cup \Omega X$ and $c(A)$ is minimal 
 among all the cuts that separate $u$ and $v $.   A cut is defined to be {\it thin} if it is thin with respect to $u, v$ for some $u,v    \in VX\cup \Omega X$.

A cut $A$ is defined to be {\it tight} if both $A$ and $A^*$ are connected, i.e if $x, y  \in A$ then there is a path joining $x,y$ whose vertices
are all in $A$, and similarly for $A^*$.

\begin {prop}  A thin cut is tight.
\end {prop}
\begin {proof}    Let $A$ be thin with respect to $u,v$.    It is easy to see that if $A$ separates $u,v$ then some component $C$ of $A$ or $A^*$ must separate $u,v$.      If $C$ is a component of $A$, then $\delta C \subset \delta A$ and if  $C, D$ are distinct components then $\delta C$ and $\delta D$ are disjoint.   Thus if $A$ is thin then $C =A$.   The result follows
\end {proof}
It is shown in \cite {[D2]} that there are only finitely many tight cuts $C$ with a fixed capacity  such that $\delta C$ contains a particular edge.
The proof of this in \cite {thomassen1993} is neater and it is reproduced here for completeness.   By replacing each edge with capacity $c(e)$, by $c(e)$ edges joining the same pair of vertices, we can assume that every edge has capacity one.
\begin {prop}\label {tight} For any $e \in EX$, there are only finitely many tight cuts $A$ with $|\delta A| = c(A) = k$ such that $e \in \delta A$.
\end {prop}
\begin {proof}  The proof is by induction on $k$.  For $k=1$ there is nothing to prove.  So assume $k > 1$.    We can assume that $e = xy$ is in some tight $k$-cut,  i.e. a cut $A$ such that $\delta A$ has $k$ edges.  Hence $X - e$ has a path $P$ from $x$ to $y$.   Now every tight $k$ cut that contains $e$ also contains an edge of $P$.   By the induction hypothesis there are only finitely many tight $(k-1)$-cuts in $X-e$ containing an edge of $P$, and we are done.
\end {proof}

If $A, B$ are cuts, then the sets  $A\cap B, A^*\cap B, A^*\cap B, A^*\cap B, A\cap B^*$ are also cuts.  These sets are called the $\it corners $ of
$A, B$.   This term is suggested by Figure \ref {Cuts} .   Two corners are called opposite or adjacent as suggested in  this figure.
We say two cuts  $A, B$ are {\it nested} if  one $A\cap B, A^*\cap B, A^*\cap B, A^*\cap B, A\cap B^*$ is empty.
  A set $\ce $ of cuts is said to be nested if any two elements of $\ce $ are nested.
  
Two cuts which are not nested are said to {\it cross}.
  \begin{lemma} Let $B$ be a cut.      There are only finitely many tight cuts $A$ with capacity $n$ that cross $B$. 
\end {lemma}
\begin {proof}    Let $F$ be a finite connected  subgraph of $X$  that contains $\delta B$.   If 
$A$ crosses $B$, then both $A$ and $A^*$ contain a vertex of $\delta B$.   Hence $F$ contains an edge of $\delta A$.   The lemma follows from
Proposition \ref {tight}.

\end{proof}

We consider sets  $\ce $ satisfying the following conditions:-
\begin{itemize}
\item [(i)] If $A \in \ce$, then $A^* \in \ce$.

\item [(ii)]  The set $\ce $ is nested.
\item [(iii)] If $A, B \in \ce$ and $A\subset B$, then there are only finitely many $C \in \ce $ such that $A\subset C\subset B$.
\end {itemize}
The following result was first obtained explicitly in \cite {[D1]}.
\begin{theo} \label {tree} If $\ce $ is a set satisfying conditions (i) (ii) and  (iii) , then there is a tree  $T= T (\ce)$  such that the directed edge set  is $ \ce $.
\end{theo}
\begin {proof}  Consider  the set  of maps $\alpha  : \ce \rightarrow \Z _2$ satisfying the following
\begin{itemize}
\item [(a)]  If $\alpha (A) = 1$, then $\alpha (A^*) = 0.$  If $\alpha (A) = 0$, then $\alpha (A^*) = 1$.
\item [(b)]  If  $\alpha (A) =1$ and $A\subset B$, then $\alpha  (B) =1$.

\end{itemize}

The vertex set of $T$ will be a subset of the set of all maps satisfying (a) and (b).

Put $ET = \ce $ and for  $A \in \ce$, put $\iota A =  \alpha  $ where $\alpha  (B) = 1 $ if $A\subseteq B$ or if  $A^* \subset B$.   Put $\tau A = \iota A^*$
Thus $\tau A = \beta $ where $\beta (B) = 1$ if $A^* \subseteq B$ or if $A\subset B$.
Then $\iota A $ and $\tau A$ take the same value on every $B$ except if $B = A$ or $B =A^*$.
We define $VT$ to be the set of all functions which are $\iota  A$ for some $A \in \ce $.      
It is fairly easy to check that $\iota A$ satisfies conditions  (a) and (b).    If $u =\iota A$ and $v = \tau B$, then the directed edges in a path joining
$u$ and $v$ consist of the set $\{ C \in \ce | A \subseteq C \subseteq B\}$.  Using (iii) this set is totally ordered by inclusion and if is finite by (ii)  so is the unique geodesic joining $u$ and $v$.  Thus $T$ is a tree.
\end {proof}
If $\ce $ is a nested set of cuts in a graph $X$ and there is a bound on the capacity of cuts in $\ce $, then
the set of all maps $\alpha : \ce \rightarrow \Z _2$ satisfying (a) and (b) can be identified with $VT \cup \Omega T$. 
Thus a ray in $T$ corresponds to a strictly decreasing sequence $E_1 \supset E_2 \cup \dots $ of cuts in $\ce $ and if we put $\alpha (E) =  1$ if for $i$  sufficiently large $E \supset E_i$.    It follows from Proposition \ref {tight} that the intersection of all the $E_i$'s is the empty set and that for any cut
$A$ for $i$ sufficiently large either $E_i \subset A$ or $E_i \subset A^*$.   If $\alpha : \ce \rightarrow \Z_2$ satisfies (a) and (b)  and there is a unique minimal  $B \in \ce $ for which $\alpha (B) =1$, then $\alpha = \iota B$.   If there is no such $B$ then we can find a  a strictly decreasing sequence $E_1 \supset E_2 \cup \dots $ of cuts in $\ce $ such that $\alpha E_i = 1$ for every $i$.  Thus $\alpha $ corresponds to a ray in $T$.

We can identify  a vertex $v$  of $X$ with a map $v : \B X \rightarrow  \Z _2$.  Thus $v(A) = 1$ if $v \in A$ and $v(A) = 0$ if $v \notin A$.
  Restricting to $\ce $ will give a vertex of $T$.  Thus there is a map $\nu :  VX \rightarrow VT$ such that $\nu (\alpha ) , \nu (\beta )$ differ only on the cuts separating $\alpha $ and $\beta $.
  
  Note that there may be vertices of $T$ which are not in the image
of $\nu $.   
Each directed edge $e$  of $X$ will give a finite directed path in $T$ consisting of those cuts $A \in \ce $ for which $\iota e \in A$ and $e \in \delta A$.
A ray in $X$ will correspond to a path in $T$ by concatenating the paths for each edge.   It may be the case that this path may back track.    It will
determine a ray in $X$ unless it visits a particular vertex of $T$ infinitely many times.    Note that because of Proposition \ref {tight} it cannot visit
two distinct vertices infinitely many times.   Thus we can extend $\nu $ so that it is a map $\nu : VX \cap \Omega X \rightarrow VT\cap  \Omega T$.

If $\ce \subset \B X$ satisfies the above conditions, then there is a tree $T(\ce )$.   If $G$ is the automorphism group of $X$ and $\ce $ is a $G$-set, then $T$ which is a $G$-tree,  is called
a {\it structure tree} for $X$.    If $T = T(\ce )$ is a structure tree for $X$,  then the  map $\nu : VX \rightarrow VT$ defined above is a $G$-map.

We now show that a structure tree determines a decomposition of the graph $X$,  in the same way that a group $G$ acting on a tree determines a decomposition of the group $G$.
 
 Let $T= T(\ce )$ be a structure tree for $X$ corresponding to a nested set of tight cuts $\ce $ satisfying the finite interval condition (iii).
 Let $v \in VT$.

Let $\nu : VX \rightarrow VT$ be the map defined earlier.
We define a graph $X _v$ as follows.   We take $EX _v$ to be the edges which lie in  $\delta C$ for some edge $C$ of $T$ incident with $v$, together with those edges $e$ such that $\nu $ maps both vertices of $e$ to $v$.
We take $VX_v$ to be the set of vertices of these edges, but we identify vertices $x,y$ if  they both lie in $C^*$ when $C$ has initial vertex $v$.

Each vertex $x$ of $X$ for which $\nu x = v$ is a vertex of $VX _v$.  Such a vertex is called a $\nu$-vertex, but there may be no such vertices.    There is another vertex of $X_v$
for each cut $C$ with initial vertex $v$ and this vertex is obtained by identifying all the vertices of $\delta C$ that are in $C^*$.  Such a vertex is called a $\rho $-vertex.   A $\rho $-vertex has degree $|\delta C|$.

It is fairly easy to see that $X_v$ is connected.   Thus any two vertices of $X$ are joined by a path in $X$.    If $x,y$ are two vertices of $X$ that 
become vertices of $X_v$ after carrying out the identifications just described, then the path in $X$ will become a path $p$ in $X_v$ if we delete any
edges that are not in $X_v$.   Here we use the fact that $C^*$ is connected, and when $p$ enters $C^*$ at vertex $w$  it must leave $C^*$ at a vertex $w'$ that is identified with $w$ in $X_v$.

In a similar way a ray in $X_v$ corresponds to a ray in $X$.      If the ray passes through a vertex corresponding to the cut $C$, then the two incident
edges will both lie in $\delta C$.  There will be a path in $C^*$ joining the corresponding vertices before they are identified.    For each such
vertex that is visited by the ray we can add in this path to obtain a ray in $X$.    This ray will belong to an end $\omega \in \Omega X$ such that
$\nu \omega  = v$.    

We regard the graphs $X_v$ for each $v \in VT$ as the factors in the decomposition for $X$.   We now describe how they fit together to give $X$.
For each edge $e \in EX$ there are only finitely many $E \in \ce $ such that $e \in \delta E$.     These edges form the edges of the geodesic in
$T$ joining $\nu u$ and $\nu v$ where $u,v$ are the vertices of $e$.  Suppose there are $k(e)$ such edges.   Now form a new graph $X'$ in which
each edge $e$ is subdivided into $k(e)$ edges.     We can extend $\nu : VX \rightarrow VT$ to a map also denoted $\nu : VX' \rightarrow VT$ which
can be extended to a graph morphism. It will now be the case that $\nu $ is surjective.
For each cut $A \in \ce$, there is a cut $A' \in \B X'$ in which $\delta A'$ consists of those edges of $X'$ that are mapped to the edge $A$ of $ET$ under the extended morphism  $\nu : X'\rightarrow T$.  We have that $A' \cap VX = A$ and $|\delta A' |= |\delta A|$.

Clearly we have a nested set of cuts $\ce '$ that will be the edge set of a structure tree $T$ which is isomorphic to $T$, and can be identified with
$T$ in a natural way.   For $v \in VT$ the graph $X'_v$ will be a subdivision of the graph $X_v$.   Each edge of $X_v$ that joins two $\rho $-vertices
is subdivided into two edges in which the centre vertex is a $\nu $-vertex in $X'$.

It is easier to use $X'$ rather than $X$ to understand the structure of the graphs $X_v$.   This is because every edge of $X'$ lies in at most one
$\delta A$ for $A \in \ce '$ and if $u,v$ are the vertices of $A \in ET = ET'$, then $X_u$ and $X_v$ are the only factors in the decomposition of $X$
that contain the edge $e$.    If $G$ is a group acting  on $X$ then it will also act on $T$.    For $v \in VT$ the stabiliser $G_v$ will act on $X_v$.
Two directed edges of $X'_v$ will lie in the  same $G$-orbit if and only if they lie in the same $G_v$-orbit.      If $u,v$ are distinct vertices of $T$, then
edge sets  of $X'_v$ and $X'_u$ are disjoint unless $u.v$ are adjacent in $T$ in which case the intersection consists of the edges that map to the edge with vertices $\nu u$ and $\nu v$ in $T$.   It follows that 
 if $X$ has finitely many $G$-orbits, then
each $X_v$ has finitely many $G_v$-orbits.

We have a decomposition of $X$ in which the factors are $X_v$.    It is possible to recover the graph $X$ from its factors.   It is easy to see how
to get the graph $X'$ and then one removes any vertices of degree two, to get $X$ (or $X$ with vertices of degree two removed).
\begin {lemma}\label {thincorner}  If $A, B$ are crossing thin cuts, with $c(A) = m, \ c(B) =n$, then  after relabelling $A$ as $A^*$ and $B$ as $B^*$  if necessary, both  $A\cap B^*, A^*\cap B$ are thin cuts with capacities  $m, n$ respectively . 
\end{lemma}    
\begin {proof}   We refer to Fig 3.   Suppose $m\leq n$.  Suppose $A$ is thin with respect to $x, y$ with $x \in A$  and $B$ is thin with respect to $x', y'$ with $x' \in B$.
After possible relabelling we can assume  $a\leq b, \ c\leq d$.   If $a < c$ then $c(A\cap B) < n$ and $c(A^*\cap B) < n$ and so 
$B$ is not thin since one of these two corners separates $x'$ and $ y'$.   If $c <a $, then $A$ is not thin.  Hence  $a = c$.  If $a < b$, then $c(A\cap B) = 2a + f  < a +e + f + b = m$, and so $x \in A\cap B^*$ and $c(A\cap B^*)= a + e + c  = m$   and  $f = 0$ .   Also  $c(A^*\cap B) = a + e + d = n$,   and $x' \in A^*\cap B $  and so it is thin, and we are done.    
If $a = b$, then $m  = 2a + e  + f  \leq a +e + f + d = n$ and so $b\leq d$.  Thus  $a= b=c \leq  d$.  If $e$ is not $0$, then $c(A\cap B) < m, c(A^*\cap B^*) < n$ and the lemma follows easily.  If $e =0$ and $f \not= 0$, then $c(A\cap B^*) < m , c(A^*\cap B) <n$ and the lemma follows if we relabel
$A$ as $A^*$.   
But if $e = f= 0$, then $c(A\cap B) = c(A\cap B^*) = m$ and $c(A^*\cap B) = c(A^*\cap B^*) = n$.      In this situation it is not possible that two adjacent 
corners of $A, B$ are not thin.
   Thus for one pair of opposite corners we have that both corners are thin. By relabelling we can assume these
corners are $A\cap B^*$ and $A^*\cap B$  and the lemma is proved.   

\end {proof}

\begin {lemma}\label {corner}   Let  $A, B, C$  be cuts .    
\begin {itemize}
\item [(i)] Let $A, B$ be not nested  and let $C$ be nested with both $A$ and $B$, then $C$ is nested with every
corner of $A, B$.  
\item [(ii)]If $C$ is nested with $A$, then $C$ is nested with two adjacent corners of $A$ and $B$.
\end {itemize}
\end {lemma}

\begin{proof} 
For (i) by possibly  relabelling   $A$ as $A^*$ and/or $B$ as $B^*$ and/or $C$ as $C^*$  we can assume either
\begin{itemize}
\item [(a)]  $C \subset A$ and $C\subset B$
or 
\item[(b)] $C\subset A$ and $C^*\subset B$.    
\end{itemize}
If (a) then $C\subset A\cap B$ and $C$ is contained in the complement of each of the other corners.
If (b), then   $B^*\subset C\subset A$, and so $A, B$ are nested, which contradicts our hypothesis.

For (ii) if $A \subset C$, then $A\cap B \subset C$ and $A\cap B^* \subset C$.
\end{proof}

Let $\C $ be a set of  cuts 
Let $A$ be a cut and let $M(A, \C)$ be the set of  cuts in $\C $  which are not  nested with $A$. Set $\mu  (A,\C) = |M(A, \C)|$.

\begin{lemma}\label{corners_equality}
Let $\C $ be a nested set of  tight cuts.   Let $A$ be a tight cut which is not nested with some $B\in \C$.
Let $\mu (A) = \mu (A, \C)$ be the number of cuts in $\C$ that are not nested with $A$. then

\[\mu(A\cap B, \C ) + \mu (A\cap B^*, \C)  <  \mu (A,\C) .\  \ \ \  \]
\end{lemma}

\begin{proof}
  If $C \in \C$ is nested with $A$, then it is nested with both $A$ and $B$ and so it is nested with $A\cap B$ and $A\cap B^*$ by Lemma \ref{corner}.
If $C$ is not nested with $A$, then it  must be nested with one of $A\cap B$ and $A\cap B^*$.   For if, say, $C \subset B$ then $B^* \subset C^*$ and so $A\cap B^* \subset C^*$.
Thus $C$ is not nested with at most one of  $A\cap B$ and $A\cap B^*$  and the lemma follows, since $B$ is counted on the right but not on the left.
\end{proof}

Let $\C _n$ be the  set of thin cuts with capacity $n$.

\begin {theo} \label {cE} There is a uniquely defined nested set of thin cuts $\ce $ in which $\ce _n = \{ E \in \ce | c(E) \leq n\}$ constructed  inductively  as follows:-

$\ce _1 = \C _1$.

If     $\D _n =  \{  A \in \C _n | \mu (A, \ce _{n-1}) = 0\}$,
then $\ce _n  = \ce _{n-1}\cup \D _n' $, where $\D _n'$ consists of all those cuts  $D \in \D_n$ satisfying 
\begin{itemize}
\item [(*)] $D$ is thin with respect to some
$u,v \in  VX \cup \Omega X$  and $\mu (D, \D _n)$ takes the minimal value among all $D \in \D _n$ that are thin with respect to $u,v$.
\end {itemize}
\end {theo}
A cut in $\D_n$ satisfying (*) is said to be {\it optimally nested with respect to $u,v$}.

\begin {proof}  This is an argument from \cite {[DW]}.  Put $\mu (A) = \mu (A, \D_n)$.   Let $A, B \in \D '$ be not nested. Each corner of $A, B$ is nested with every $e \in \ce _{n-1}$  by Lemma \ref {corner}.  Suppose $A$ is optimally nested  with respect
to  $x, y$ and $B$ is optimally nested with respect to $x', y'$.   Here $x, y, x',y'$ are elements of $VX \cup \Omega X$.   Each of $x, y, x', y'$ determines a corner   of
$A, B$.   There are two possibilities.
\begin {itemize}
\item [(i)]  The sets $x, y$ determine opposite corners, and $x', y'$ determine the other two corners.
\item [(ii)]  There is a pair  of  opposite corners such that one corner is determined by one of $x, y$ and the opposite corner
is determined by one of $x', y'$.  
\end {itemize}
 
In case (i) $A$ and $B$ separate both pairs $x, y$ and $x',y'$.   Since $A, B$ are optimally nested with respect to
$x, y$ and $x',y'$,  we have  $\mu (A) = \mu(B)$.   But now  both $A\cap B$ and $A^*\cap B^*$ separate $x,y$ say  and  
$c(A\cap B)  = c(A^*\cap B^*) =n$ by Lemma \ref {thincorner}   so that both the corners are in $\D _n$,  and
$\mu (A\cap B) +\mu (A^*\cap B^*) < \mu (A)+\mu (B) =2\mu (A)$, by Lemma \ref {corner},  since if an element of $\C _n$ is not nested with both $A\cap B$ and $A^*\cap B^*$ it is not nested with both
$A$ and $B$ and if it is not nested with one of $A\cap B, A^*\cap B^*$ then it is not nested with one of $A$ and $B$.   The strict equality follows because 
$A \in \D _n$ separates $x,y$ and is not nested with $B$ but both $A\cap B, A^*\cap B^*$ are nested  with $A$ and $B$.
Since both $A\cap B, A^*\cap B^*$ separate $x$ and $y$ we have a contradiction.

In case (ii) suppose these corners are $A\cap B$ and $A^*\cap B^*$, and that $x \in A\cap B,  y' \in  A^*\cap B^*$.    But then $A\cap B$ separates $x$ and $y$ and $A^*\cap B^*$ separates $x'$ and $y'$.    
Since $A$ is optimally nested with respect to $x$ and $y$ we have $\mu (A\cap B) \geq \mu (A)$ and since $B$ is optimally nested with respect to 
$x'$ and $y'$ we have $\mu (A^*\cap B^*) \geq \mu (B)$.   But it follows from  Lemma  \ref {corners_equality}
 that $\mu (A\cap B) + \mu(A^*\cap B^*) < \mu (A) + \mu(B)$ and so we have a
contradiction.    Thus $\D _n '$ is a nested set  and the proof is complete.

Note that $\ce $ is uniquely defined, since no choices are made in its construction.   This is very important in applications.
It means that $\ce $ is invariant under the automorphism group of the graph.

\end {proof}   
Recall that $\B X$ is the Boolean ring  of all cuts.   Let $\B _nX$ be the ring generated by all cuts $A$ such that $c(A) \leq n$.

\begin {theo} \label {accessible}  For every $u,v \in VX \cup \Omega X$  that can be separated by a cut,  $\ce$ contains a cut $A$ that is thin with respect to $u,v$.
The set $\ce _n$ generates $\B _nX$ and is the directed edge set of a tree.
\end{theo}

\begin {proof}
To prove the first statement we need to show that for every $u,v$  there is a cut $A$ that is thin with respect to $u,v$  that is nested with every
cut $E \in \ce _{n-1}$, where $n = c(A)$, i.e  we need to show that there is a cut in $\D _n$ that is thin with respect to $u,v$.
We know that there is a cut $B$ that is thin with respect to $u,v$.    Let $k = \mu (B, \ce _{n-1})$.  If $k = 0$ then we take $A = B$.
If $k \geq 1$,  then let $C \in \ce_{n-1}$ cross $B$.
We know from Lemma \ref {thincorner} that one of the corners (say $B\cap C$  of $B, C$ is thin with respect to $u,v$.   
By Lemma \ref {corners_equality} we have $\mu (B\cap C) +\mu (B\cap C^*) < k$.  Thus if $\mu (B) =k > 0$ we can find a cut $B\cap C$ which is thin with respect
fo $u,v$ for which $\mu (B\cap C) < k$.   Thus there must be a cut $A$ which is thin with respect to $u,v$ for which $\mu (A) =0$.

  Let $B_n$ be the subring of $\B $ generated by $\ce _n$.   Clearly $B _n \subseteq \B_n$ the subring
generated by all cuts with capacity at most $n$.     We want to show that $B_n = \B_n$.  Let $A$ be a cut with $c(A) =n$.
We show that $A \in B_n$ by induction on $\mu (A) = \mu (A, \ce _n)$.   
Suppose that $\mu (A) =0$ so that $A$ is nested with every cut in $\ce _n$.   In this case if $A\notin \ce _n$ then $A$ determines a vertex $v$ of $T= T(\ce _n)$.    Thus we define 

$v  : \ce _n \rightarrow \Z _2,  v E = 0, {\rm if} \  E \subset A,  {\rm or}\  E \subset A^*,   v E = 1, {\rm if} \  E^* \subset A,  {\rm or}\  E^* \subset A^*$,   

We have that $\tau E = v$ if $E \subset A$ and $E$ is a maximal element of $\ce _n$ with this property.

Let $X_v$ be the graph defined earlier.   

The cut $A$ will become a cut $A_v$ in $X_v$.    Thus $A_v$ will consist of all $\nu$-vertices $x$ such that $x \in A$ and also all $\rho $-vertices
$y$ corresponding to cuts $C \in \ce _n$ such that $C^* \subset A$.  It will then be the case that $\delta A = \delta A_v$.
Now $A \in \B _n$ if and only if either $A_v$ or $A_v ^*$ consists of finitely many $\rho $-vertices.
If this is not the case then both $A_v$ and $A_v^*$ either contain infinitely many $\rho $-vertices of at least one $\nu $-vertex.
However if say $A_v$ consists only of  infinitely many $\rho $-vertices, then $A_v$ is a connected subset of $X_v$ consisting of  infinitely many vertices of bounded degree.  Such a set must contain the vertices of a ray by K\" onig's Lemma.    It follows that if $A$ is not in $\B _n$ then it must separate two
elements of $VX \cup \Omega X$.  But these two elements must be separated by an element of $\ce _n$ which is not the case, and so we have a
contradiction.   Thus $A_v$ or $A_v ^*$ consists of finitely many $\rho$-vertices and so $A \in B_n$.

If $\mu (A) >0$ and $E \in \ce _n$ is not nested with $A$,   then $A = A\cap E + A\cap E^*$,   and both $A\cap E$ and $A\cap E^*$ are in
$B_n$ by induction on $k = \mu (A)$.    The theorem is proved.

\end {proof}

  If  every pair $x, y \in VX \cup \Omega X$ that can
can be separated by a cut can be separated by a cut in $\ce _n$, then $\B X = \B _nX$.     If $\B X= \B _n X$  and $\nu x = \nu y =v$ then
$x,y$ cannot be separated in $X$ or $X_v$.   This means that either  both $x,y$ and every other vertex of $X_v$  have  infinite degree and $X_v$ has one end or $x = y$ and $X_v$
contains at most one   $\nu $-vertex.      If $v$ is not in the image of $\nu $, and $\B X = \B_n X$, then $X_v$ will be a one ended graph in which each
vertex has degree at most $n$.

Thomassen and Woess \cite{thomassen1993} define a graph to be {\it accessible} if there is an integer $n$ such that any two ends can be separated
by removing at most $n$ edges.    Alternative ways of defining accessibility are suggested by Theorem \ref {accessible}.
\begin {defi}   A graph $X$ is said to be $\B $-accessible if $\B X = \B _nX$ for some $n$.   A graph $X$ is said to be $\ce $-accessible if
$\ce $ satisfies the finite interval condition (iii) of Theorem \ref {tree}.
\end {defi}

A graph $X$ is $\B $-accessible if and only every pair $x,y \in VX\cup \Omega X$ that can be  separated by a cut  can be  separated
by a cut in $\ce _n$.   There is then a structure tree $T_n$ and a map $\nu : VX \cup \Omega X \rightarrow VT_n \cup \Omega T_n$ such
that $\nu x \not= \nu y $ if and only if $x, y$ are  not  separated by any  cut .

A graph $X$ is $\ce $-accessible if and only if $\ce $ is the directed edge set of a structure tree $T$ for which there is a map $\nu : VX\cup \Omega X \rightarrow VT\cup \Omega T$ in which $\nu x = \nu y$ if and only if $x,y$ are not separated by any cut $A$.
If $X$ is $\ce $-accessible, then for every $v \in VT$ the graph $X_v$  has the following structure:-

\begin {itemize}

\item [ ]  $X_v$ has at most  one end.    Every $\rho $-vertex has finite degree.    There may be no other vertices.  It there are other 
vertices, which will be  $\nu $-vertices, then there  is at most one such vertex of finite degree more than $2$.    Any $\nu $-vertex of degree $2$ lies
between two   $\rho $-vertices.  No two  $\nu $-vertices of infinite degree can be separated by a  cut in either $X$ or $X_v$.

\end {itemize}
Clearly it follows from Theorem \ref {maintheoremB} that $X$ is accessible if  $\B X = \B_nX$ for some $n$.
If every vertex of $X$ has bounded degree then $\B X = \B _nX$ for some $n$ if and only if $X$ is $\ce$-accessible,  and all three definitions of accessible are equivalent.      Wall \cite {[W]} defined a finitely generated group to be accessible if any process of successively splitting the group over finite 
subgroups eventually terminates with factors which are finite or one-ended.    By Bass-Serre theory this is equivalent to saying that the group has
an action on a tree with every edge group finite and every vertex group is  finite or one-ended.    

As proved in \cite {thomassen1993}  a finitely generated group is accessible if and only if its Cayley graph (with respect to a finite generating set) is accessible.   I proved in \cite {[D2]} that finitely presented groups are accessible and in \cite {Dun} I gave an example of a finitely generated group that is not accessible.   Thus there are vertex transitive locally finite graphs that are not accessible.

It is fairly easy to construct graphs that are $\ce $-accessible but which are not $\B $-accessible or accessible.
A graph that is $\B$-accessible is both accessible and $\ce $-accessible.

\begin {theo} A locally finite accessible graph is  $\ce $-accessible.   
\end {theo}
\begin {proof}  Suppose that the any two ends of $X$ are separated by an $n$-cut, so that $X$ is accessible.
Let $T = T_n$ be the structure tree with edge set $\ce _n$.     For each $v \in VT$,  $X_v$ has at most one end.
We have to show that any edge of $X$ lies in finitely many $\delta A$ for $A \in \ce $.   We know that any edge of $X'$ lies in at most two graphs $X'_v$ and so any edge of $X$ lies in finitely many $X_v$.    It therefore suffices to show that any edge of $X_v$ lies in finitely $\delta A$ for 
$A\in \ce$.   Such an $A$ can be regarded as a cut in $X_v$.    Thus it suffices to prove the theorem when $X$ has at most one end.
Clearly the result is true if $X$ is finite.    By K\" onig's Lemma,  if $X$ is infinite than it has a ray and so in our case it has one end.
There will be an infinite sequence of elements $A _i \in \ce $ such that $A_{i+1} \subset A_i$ whose intersection is empty.   Since $X$ has one end
$A_i^*$ is finite.     If $A_i\in \ce _{n(i)}$,  then the edges in $T_{n(i)}$ that separate vertices  of $A_i^*$ form a finite subtree and this will be a subtree
of every tree $T_m$ for $m > n(i)$.   it follows that every edge of $X$ lies in finitely many $\delta A$ for $A \in \ce $.

\end {proof}

I think it ought to be possible to  drop the locally finite condition in the last theorem.

In \cite {[D3]} I showed that a vertex transitive locally finite planar graph is accessible.  
A locally finite graph $X$ has a Freudenthal compactification $\F (X)$  In which distinct ends correspond to distinct points.    Richter and Thomassen
\cite {[RT]} show that  if $X$ is  connected locally finite planar graph then $\F (X)$ can be embedded in $S^2$ and this embedding has a uniqueness
property if $X$ is $3$-connected.  If $A$ is a  tight cut in such a graph $X$, then there if a simple closed curve in  $S^2$ that intersects  $\F (X)$ in 
a finite set of points consisting of a single point in each edge of $\delta A$.     The set $\ce $ will correspond to a set of non-intersecting simple closed
curves.  The graph $X$ will be $\ce $-accessible if and only if every edge of $X$ intersects finitely many of the simple closed curves.   The structure tree corresponding to $\ce $ will then be the dual graph to the set of curves.

It would be interesting to know if every planar graph is $\ce $-accessible.
If this was the case then  a planar graph of bounded vertex degree
would be  $\B$-accessible.
A locally finite $\ce $-accessible graph has a unique decomposition in which the factors are one ended or finite.


\begin{figure}[ht!]

\centering
\begin{tikzpicture}[scale=.5]

    \draw (0,4)--(8,4);
    \draw (4,0)--(4,8);

\draw (3.6, 7.8) node  {$A$};
\draw (4.6,7.8) node  {$A^*$};
\draw (.4, 4.4) node  {$B$} ;
\draw (.4,3.6) node  {$B^*$} ;
\draw (6,6) node  {$A^*\cap B$};
\draw (2,6) node  {$A\cap B$};
\draw (2,2) node  {$A\cap B^*$};
\draw (6,2) node  {$A^*\cap B^*$};

    \draw (10,5)--(13,5)-- (13, 8);
       \draw (10,3)--(13,3)-- (13, 0);
\draw (18,5)--(15,5)-- (15, 8);
       \draw (18,3)--(15,3)-- (15, 0);

\draw [dashed] (13,3)--(15,5) ;
\draw [dashed] (15,3)--(13,5) ;

\draw (17,6) node  {$A^*\cap B$};
\draw (11,6) node  {$A\cap B$};
\draw (11,2) node  {$A\cap B^*$};
\draw (17,2) node  {$A^*\cap B^*$};

\draw (14,7) node  {$a$};
\draw (14,1) node  {$b$};
\draw (11, 4) node  {$c$};
\draw (17, 4) node  {$d$};
\draw (13.5, 3.5) node  {$e$};
\draw (14.5, 3.5) node  {$f$};

  \end{tikzpicture}

  \caption{Crossing cuts}\label{Cuts}
\end{figure}
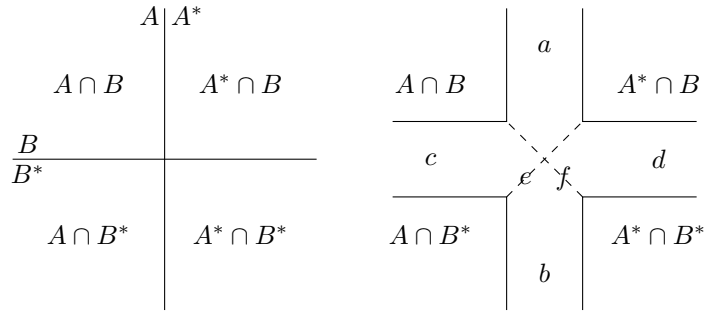

Combining the two previous theorems we have the following.

\begin {theo}\label {maintheoremB}  Let $N(X)$ be a network in which $X$ is an arbitrary connected graph.  For each $n >0$, there is a network $N(T_n)$ based on a tree  $T=T_n$ and a map $\nu : VX\cup \Omega X \rightarrow VT\cup \Omega T$, such
that  $\nu (VX ) \subset VT$ and  $\nu x = \nu y$ for any $x, y \in VX \cup \Omega X$ if and only if $x, y$ are not separated by a cut $A$ with $c(A) \leq n$. 

The network $N(T_n)$ is canonically determined and is invariant under the automorphism group of $N(X)$.
\end {theo}

For a finite network Theorem \ref {maintheoremB} reduces to Theorem \ref {maintheorem}.    For a finite network the structure tree
of Theorem \ref {maintheorem} will become a Gomory-Hu tree by contracting certain edges.    Thus a Gomory-Hu tree has the same properties
as our tree except that the map $\nu : VX \rightarrow VT$ is a bijection.      One obtains a Gomory-Hu tree from our structure tree by choosing
for each vertex $v \in VT$ that is not in the image of $\nu $ an incident edge of maximal capacity and then contracting all those edges.
If we choose the edges to contract in the way just described one will preserve the property that for any $s,t  \in VX$, there is a minimal cut separating
$s, t$ in the geodesic joining $\nu s, \nu t$ in $T$.
Note that while the tree of Theorem \ref {maintheorem} is uniquely determined there may be more than one Gomory-Hu tree.
This has already been noted in the tree of Fig \ref {fig:Tree}.

Let $n$ be the smallest capacity of a cut in $N(X)$.   It can be seen fairly easily from Fig \ref {Cuts} that if   $A, B$ are cuts with $c(A) = c(B)  =n$ and $A$ is not nested with $B$ , then $n = 2m$ and $\delta A$ partitions $\delta A = \delta (A\cap B)\cup \delta (A\cap B^*)$,  where each
of $\delta (A\cap B)$ and $\delta (A\cap B^*)$ contain $m$ edges.   If $C$ is also a cut with capacity $n$ that is not nested with $A$ then
one can show that the partition of $\delta A$ given by $C$ is the same as that corresponding to $B$.   This result is crucial in the cactus representation of mincuts by Dinits, Karzanov and Lomonosov \cite {[DKL]}.   A cactus is a simple graph in which every edge lies in at most one  cycle.    In the cactus representation  there is a cactus $K$ and a mincut in $N(X)$ corresponds to a tight cut in $K$ with capacity at most $2$.
In our notation the mincuts in $N(X)$ are the elements of $\C _n= \D_n$ and the elements of $\D _n ' = \ce _n$ correspond to the tight cuts $E$ in
$K$ in which $\delta E$ has one edge or it consists of adjacent edges of a cycle.

Evangelidou and  Papasoglu \cite {[EP]} use a similar cactus argument but for minimal cuts separating ends of a graph to give a proof of Stallings' Theorem.

We illustrate our results with some simple examples.

In the example of Fig \ref {ladder}, the graph $X$ is an infinite ladder, and every edge has capacity one.  The two ends of $X$ and some vertices are separated in $T_2$ and all ends and vertices in $T_3$.
The vertices of $T_3$ on the red central line are not in the image of $\nu$.
Note the cuts $A, A^*$ given by the brown dashed line are not thin with respect to the two ends of the graph since $c(A) = 3$ and the two ends are
separated by a cut with capacity $2$.   However $A$ is thin with respect to the two vertices $u,v$ of one rung of the ladder.   However $A$
is not optimally nested with respect to $u,v$ since it is not nested with the cut $\alpha A$ where alpha is the automorphism swapping the two sides
of the ladder.  The cuts given by the blue dashed lines are nested with every thin cut and so are optimally nested with respect to any pair of vertices that they separate.

\begin{figure}[htbp]
\centering
\begin{tikzpicture}[scale=.7]
\draw [thick]  (0,0) -- (10, 0)  ;
\draw [thick](0,2) -- (10, 2)  ;
\draw [thick, red]  (12,1) -- (22, 1)  ;
\draw [thick] (2,0) --(2,2) ;
\draw [thick] (4,0) --(4,2) ;
\draw [thick](6,0) --(6,2) ;
\draw [thick] (8,0) --(8,2) ;
\draw [very thick, dashed, brown]  (3.4, 3) -- (3.4, 1.6) --( 4.6, .4) -- (4.6, -1) ;
\draw [dashed, red] (7,-1) --(7,3) ;
\draw [dashed, red] (5,-1) --(5,3) ;
\draw [dashed, red] (3,-1) --(3,3) ;
\draw [dashed, red] (1,-1) --(1,3) ;
\draw [dashed, red] (9,-1) --(9,3) ;

\draw [dashed, blue ] (1.7,3) --(1.7,1.7)--( 2.3, 1.7) -- (2.3, 3) ;
\draw [dashed, blue ] (3.7,3) --(3.7,1.7)--( 4.3, 1.7) -- (4.3, 3) ;
\draw [dashed, blue ] (5.7,3) --(5.7,1.7)--( 6.3, 1.7) -- (6.3, 3) ;
\draw [dashed, blue ] (7.7,3) --(7.7,1.7)--( 8.3, 1.7) -- (8.3, 3) ;

\draw [dashed, blue ] (1.7,-1) --(1.7,.3)--( 2.3, .3) -- (2.3, -1) ;
\draw [dashed, blue ] (3.7,-1) --(3.7,.3)--( 4.3, .3) -- (4.3, -1) ;
\draw [dashed, blue ] (5.7,-1) --(5.7,.3)--( 6.3, .3) -- (6.3, -1) ;
\draw [dashed, blue ] (7.7,-1) --(7.7,.3)--( 8.3, .3) -- (8.3, -1) ;
\filldraw (17, 4) circle (3pt);
\draw (17,5) node  {$T_1$};
\draw (17,2) node  {$T_2$};
\draw (17,-1) node  {$T_3$};
\filldraw (17, 4) circle (3pt);
\filldraw (14, 1) circle (3pt);
\filldraw (16,1) circle (3pt);
\filldraw (18,1) circle (3pt);
\filldraw (20,1) circle (3pt);
\draw [thick, red]  (12,1) -- (22, 1)  ;
\draw [thick, red]  (12,-2) -- (22, -2)  ;
\draw [thick,blue] (14,-3) --(14,-1) ;

\draw [thick,blue] (16,-3) --(16,-1) ;
\draw [thick,blue] (18,-3) --(18,-1) ;
\draw [thick,blue] (20,-3) --(20,-1) ;
\filldraw (14, -1) circle (3pt);
\filldraw (16,-1) circle (3pt);
\filldraw (18,-1) circle (3pt);
\filldraw (20,-1) circle (3pt);

\filldraw (14, -2) circle (3pt);
\filldraw (16,-2) circle (3pt);
\filldraw (18,-2) circle (3pt);
\filldraw (20,-2) circle (3pt);
\filldraw (14, -3) circle (3pt);
\filldraw (16,-3) circle (3pt);
\filldraw (18,-3) circle (3pt);
\filldraw (20,-3) circle (3pt);
\draw (4,-.2) node  {$u$};
\draw (4, 2.2) node  {$v$};
\draw (16,-3.4) node  {$\nu u$};
\draw (16, -.6) node  {$\nu v$};
\end{tikzpicture}

\caption {Cutting Up a Ladder} \label{ladder}
\end{figure}
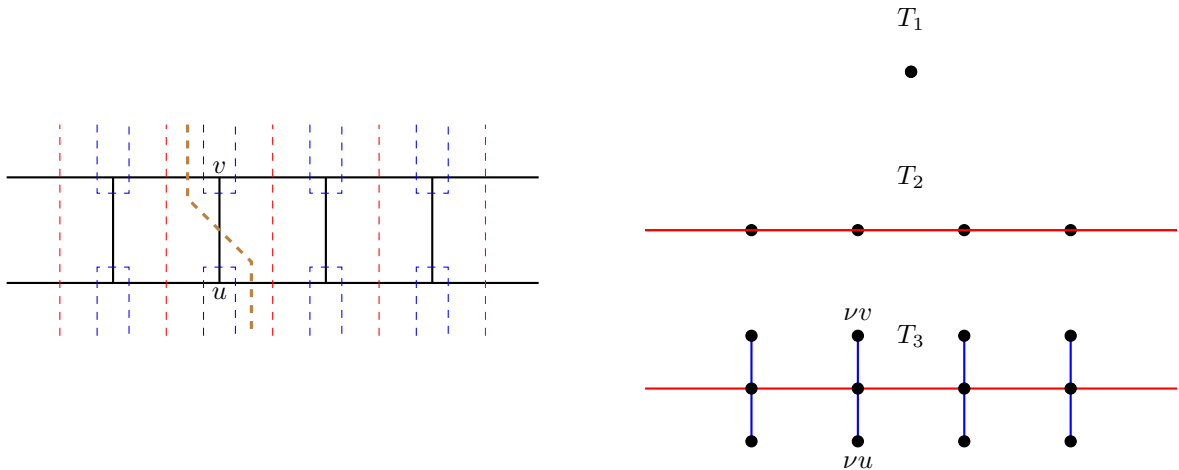

\eject

In the example of Fig \ref {ladder2}, all the  vertices are separated in $T_4$ and all ends and vertices in $T_5$.  There is a  vertex  of infinite degree in each of $T_3$ and $T_4$.

\begin{figure}[htbp]
\centering
\begin{tikzpicture}[xscale=.7, yscale = .8]
\draw [thick]  (0,1) -- (10, 1)  ;
\draw [thick]  (0,3) -- (10, 3)  ;
\draw [thick]  (0,-1) -- (10, -1)  ;

\draw [thick]  (0,0) -- (10, 0)  ;
\draw [thick](0,2) -- (10, 2)  ;
\draw [thick] (2,-1) --(2,3) ;
\draw [thick] (4,-1) --(4,3) ;
\draw [thick](6,-1) --(6,3) ;
\draw [thick] (8,-1) --(8,3) ;
\draw [dashed, red] (7,-2) --(7,4) ;
\draw [dashed, red] (5,-2) --(5,4) ;
\draw [dashed, red] (3,-2) --(3,4) ;
\draw [dashed, red] (1,-2) --(1,4) ;
\draw [dashed, red] (9,-2) --(9,4) ;

\draw [dashed, blue ] (1.7,2.3) --(1.7,1.7)--( 2.3, 1.7) -- (2.3, 2.33)--cycle ;
\draw [dashed, blue ] (3.7,2.3) --(3.7,1.7)--( 4.3, 1.7) -- (4.3, 2.3)--cycle ;
\draw [dashed, blue ] (5.7,2.3) --(5.7,1.7)--( 6.3, 1.7) -- (6.3, 2.3)--cycle ;
\draw [dashed, blue ] (7.7,2.3) --(7.7,1.7)--( 8.3, 1.7) -- (8.3, 2.3)--cycle ;

\draw [dashed, blue ] (1.7,-.3) --(1.7,.3)--( 2.3, .3) -- (2.3, -.3)--cycle ;
\draw [dashed, blue ] (3.7,-.3) --(3.7,.3)--( 4.3, .3) -- (4.3, -.3) --cycle;
\draw [dashed, blue ] (5.7,-.3) --(5.7,.3)--( 6.3, .3) -- (6.3, -.3)--cycle ;
\draw [dashed, blue ] (7.7,-.3) --(7.7,.3)--( 8.3, .3) -- (8.3, -.3) --cycle;

\draw [dashed, blue ] (1.7,1.3) --(1.7,.7)--( 2.3, .7) -- (2.3, 1.3)--cycle ;
\draw [dashed, blue ] (3.7,1.3) --(3.7,.7)--( 4.3, .7) -- (4.3, 1.3) --cycle;
\draw [dashed, blue ] (5.7,1.3) --(5.7,.7)--( 6.3, .7) -- (6.3, 1.3)--cycle ;
\draw [dashed, blue ] (7.7,1.3) --(7.7,.7)--( 8.3, .7) -- (8.3, 1.3) --cycle;

\draw [dashed, thick, brown ] (1.7,4) --(1.7,2.7)--( 2.3, 2.7) -- (2.3, 4) ;

\draw [dashed, thick, brown ] (3.7,4) --(3.7,2.7)--( 4.3, 2.7) -- (4.3, 4) ;
\draw [dashed, thick, brown ] (5.7,4) --(5.7,2.7)--( 6.3, 2.7) -- (6.3, 4) ;
\draw [dashed, thick, brown ] (7.7,4) --(7.7,2.7)--( 8.3, 2.7) -- (8.3, 4) ;

\draw [dashed, thick, brown ] (1.7,-2) --(1.7,-.7)--( 2.3, -.7) -- (2.3, -2) ;

\draw [dashed, thick, brown ] (3.7,-2) --(3.7,-.7)--( 4.3, -.7) -- (4.3, -2) ;
\draw [dashed, thick, brown ] (5.7,-2) --(5.7,-.7)--( 6.3, -.7) -- (6.3, -2) ;
\draw [dashed, thick, brown ] (7.7,-2) --(7.7,-.7)--( 8.3, -.7) -- (8.3, -2) ;

\draw [thick, brown]   (14,2) -- (17, 3) --(15,2) ;
\draw [thick, brown]   (16,2) -- (17, 3) --(17,2) ;
\draw [thick, brown]   (18,2) -- (17, 3) --(19,2) ;
\draw [thick, brown]   (20,2) -- (17, 3) ;

\draw [thick, brown]   (14,-1) -- (17, -.25) --(15,-1) ;
\draw [thick, brown]   (16,-1) -- (17, -.25) --(17,-1) ;
\draw [thick, brown]   (18,-1) -- (17, -.25) --(19,-1) ;
\draw [thick, brown]   (20,-1) -- (17, -.25) ;

\draw [thick, blue]   (14,.5) -- (17, -.25) --(14.7,.5) ;
\draw [thick, blue]   (15.3,.5) -- (17, -.25) --(16,.5) ;
\draw [thick, blue]   (16.7,.5) -- (17, -.25) --(17.3,.5) ;
\draw [thick, blue]   (18,.5) -- (17, -.25) --(18.7,.5) ;
\draw [thick, blue]   (19.3,.5) -- (17, -.25) --(20,.5) ;

\draw [thick, brown]   (14,2) -- (17, 3) --(15,2) ;
\draw [thick, brown]   (16,2) -- (17, 3) --(17,2) ;
\draw [thick, brown]   (18,2) -- (17, 3) --(19,2) ;
\draw [thick, brown]   (20,2) -- (17, 3) ;

\filldraw (17, 3) circle (3pt);
\draw (14,5) node  {$T_1= T_2$};
\draw (13,3) node  {$T_3$};
\draw (13,0) node  {$T_4$};
\draw (13,-4) node  {$T_5$};

\filldraw (17, 5) circle (3pt);
\filldraw (14, 2) circle (3pt);
\filldraw (16,2) circle (3pt);
\filldraw (18,2) circle (3pt);
\filldraw (20,-1) circle (3pt);
\filldraw (17, -.25) circle (3pt);
\filldraw (20, 2) circle (3pt);

\filldraw (15, 2) circle (3pt);
\filldraw (17,2) circle (3pt);
\filldraw (19,2) circle (3pt);
\filldraw (19,-1) circle (3pt);

\filldraw (14, -1) circle (3pt);
\filldraw (16,-1) circle (3pt);
\filldraw (18,-1) circle (3pt);
\filldraw (15, -1) circle (3pt);
\filldraw (17,-1) circle (3pt);
\filldraw (19,-5) circle (3pt);
\filldraw (14, -5) circle (3pt);
\filldraw (16,-5) circle (3pt);
\filldraw (18,-5) circle (3pt);
\filldraw (15, -5) circle (3pt);
\filldraw (17,-5) circle (3pt);
\filldraw (19,-5) circle (3pt);

\filldraw (20,-1) circle (3pt);

\filldraw (14, -3) circle (3pt);
\filldraw (16,-3) circle (3pt);
\filldraw (18,-3) circle (3pt);
\draw [thick,  red]  (14,-4)--(20,-4) ;

\filldraw (15, -4) circle (3pt);
\filldraw (17,-4) circle (3pt);
\filldraw (19,-4) circle (3pt);

\filldraw (14, .5) circle (3pt);
\filldraw (16,.5) circle (3pt);
\filldraw (18,.5) circle (3pt);
\filldraw (20,.5) circle (3pt);

\filldraw (14.7, .5) circle (3pt);
\filldraw (16.7,.5) circle (3pt);
\filldraw (18.7,.5) circle (3pt);

\filldraw (15.3, -3) circle (3pt);
\filldraw (17.3,.-3) circle (3pt);
\filldraw (19.3,-3) circle (3pt);

\filldraw (14.7, -3) circle (3pt);
\filldraw (16.7,-3) circle (3pt);
\filldraw (18.7,-3) circle (3pt);

\filldraw (15.3, .5) circle (3pt);
\filldraw (17.3,.5) circle (3pt);
\filldraw (19.3,.5) circle (3pt);
\draw [thick, blue] (14, -3) --(15,-4) -- (15.3, -3) ;
\draw [thick, blue] (16, -3) --(17,-4) -- (17.3, -3) ;
\draw [thick, blue] (18, -3) --(19,-4) -- (19.3, -3) ;
\draw [thick, blue] (14.7, -3) --(15,-4)  ;
\draw [thick, blue] (16.7, -3) --(17,-4)  ;
\draw [thick, blue] (18.7, -3) --(19,-4)  ;
\filldraw (14, .5) circle (3pt);
\filldraw (16,.5) circle (3pt);
\filldraw (18,.5) circle (3pt);
\filldraw (20,.5) circle (3pt);

\draw [thick, brown]   (14,-5) -- (15,-4) --(15,-5) ;
\draw [thick, brown]   (16,-5) -- (17, -4) --(17,-5) ;
\draw [thick, brown]   (18,-5) -- (19, -4) --(19,-5) ;

\end{tikzpicture}
\caption {Cutting Up Another Ladder}\label {ladder2}
\end{figure}
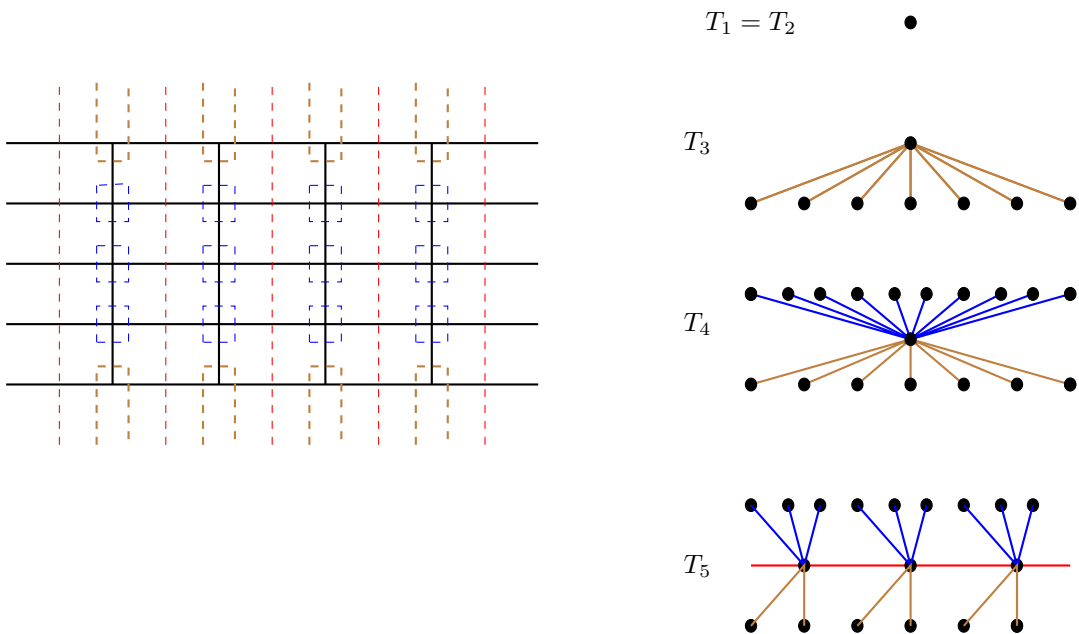

A finitely generated group $G$ is said to have more than one end,  if a Cayley graph $X = X(G, S)$ of $G$ corresponding to a finite generating set $S$ has more
than one end.
\begin {theo} [Stallings' Theorem]   If $G$ is a finitely generated group with more than one end, then $G$ has a non-trivial action on a tree $T$
with finite edge stabilizers. 
\end {theo}
\begin {proof}  Let $n$ be the smallest integer for which there a cut $A$ such that $|\delta A| = n$ which separates two ends $s, t$.
The structure tree $T = T_n$ will have the required property.    Here we use the fact that the construction of $T$ is canonical and so is invariant 
under the action of automorphisms.    Thus the action of $G$ on $X$ gives an action on $T$.    Each edge of $T$ is a cut $C$ with $|\delta C| \leq n$.
The stabilizer of $C$ will permute the edges of $\delta C$ and will therefore be finite.    We also have to show that action is non-trivial.
We know there is a cut  $D$ in $ET$ that separates a pair of ends.  For such a cut both $D$ and $D^*$ are infinite.    The action of $G$ on $X$ is
vertex transitive.    There exists $g \in G$ such that the vertices of  $g\delta D $ are contained in $D$ and an element $h\in G$ such that the vertices
of $h\delta D$ are contained in $D^*$.  It hen follows that for $x = g, x = h$ or $x =  gh$ we have $xD$ is a proper subset of $D$ or $D$ is a proper subset of $xD$.   It follows from elementary Bass-Serre theory that $x$ cannot fix a vertex of $T$.
\end {proof}
This proof is essentially that of \cite {Kroen2009}.

In any tree $T$ if $p$ is a vertex and $Q$ is a set of  unoriented edges, then there is a unique set of vertices $P$ such that $v \in P$ then the geodesic $[v,p]$
contains an odd number of edges from $Q$.   We then have $\delta P = Q$.
Note that $\B T = \B_1T$ and every element of $\B T$ is uniquely determined by the set $Q$ together with the information for a fixed  $p\in VT$ whether $p \in A$ or $p \in A^*$.    The vertex $p$ induces an orientation $\co _p$ on the set of pairs $\{ e, \bar e\}  $ of oriented edges by requiring that 
$e \in \co$ if $e$ points at $p$.      For $A \in \B T$,  $A$ is uniquely determined by $\delta A $ together with the orientation $\co _p \cap \delta A$  of the edges of $\delta A$.

In $X$ it is the case that a cut $A$ is uniquely determined by $\delta A$  together with the information for a fixed  $p  \in VX$ whether $p \in A$ or $p \in A^*$.

Since $\B _nX$ is generated by $\ce _n = ET_n$,    the cut $A$ can be expressed in terms of a finite  set of oriented edges of $T_n$.      This set is not usually uniquely
determined.    Thus if $\nu $ is not surjective, and $v$ is not in the image of $\nu $,  and  the set of edges incident with $v$ is finite, then $VX$  is the union of these elements in $\B X$.       The empty set is the intersection of the complements of these sets.     Orienting the edges incident with
$v$ towards $v$ gives the empty set and orienting them away from $v$ gives all of $VX$.
However there is a canonical way of expressing an element of $\B _nX$ in terms of the generating set $\ce _n$.      To see this let $A \in \B _nX  -  \B _{n-1}X$.    There are only finitely many $C \in \ce _n$ with which $C$ is not nested.
This number is $\mu (A, \ce _n)  = \mu (A)$.  We use induction on $\mu (A)$.   Our induction hypothesis is that there is a canonically defined way
of expressing $A$ in terms of the $\ce _n$.    Any two ways of expressing $A$ in terms of $\ce _n$ differ by an expression which gives the empty
set in terms of $\ce _n$.   Such an expression will correspond to a finite set of vertices each of which has finite degree in $T_n$ and  none of which is in the image of $\nu $.    The canonical expression is obtained if there is a unique way of saying whether or not  each such vertex is in the expression for $A$.     Thus the canonical expression for $A$ is determined by  a set of vertices of $VT$ which consists of the vertices of 
$\nu (A)$ together with a recipe for deciding for each vertex which is not in the image of $\nu $ whether it is in the expression for $A$.

Suppose $\mu (A) =0$, so that $A$ is nested with every $C \in \ce _n$, and neither $A$ nor $A ^*$ is empty.      If $A \in \ce _n$, then
this gives an obvious way of expressing $A$ in terms of the $\ce _n$.   If $A$ is not in $\ce _n$, then it corresponds to a unique vertex 
$z \in VT_n$.   Thus because $\mu (A) = 0$,  $A$ induces an orientation of the edges of  $\ce _n$.   To see this,  let $C \in \ce _n$, then  just
one of $C \subset A,  C^* \subset A,  C\subset A^*, C^* \subset A^*$ holds.   From each pair $C, C^*$ we can choose $C$ if $C \subset A$ or
$C \subset A^*$ and we choose $C^*$ if $C^*\subset A$ or $C^* \subset A^*$.    Let $\co$ be this subset of $\ce $.    Then If $C \in \co$ and $D \in \ce
$ and $D \subset C$,  then $D \in \co$.   This means that the orientation $\co$ determines a vertex $z$ in $VT _n$.  Intuitively the edges of $\co $ point at the vertex $z$.
 It can be seen that $A$ or $A^*$ will be the union of finitely many edges $E$ of $\ce _n = ET_n$,  all of which have $\tau E = z$.    If $A$ is such a union, then
we use this to express $A = C_1\cup C_2, \dots \cup C_k$.   If $A$ is not such a union, but $A^* = C_1\cup C_2, \dots \cup C_k$, then we write
$A = (C_1\cup C_2, \dots \cup C_k)^* = C_1^*\cap C_2^*\cap \dots \cap C_k^*$.        
 Note that this gives a unique way of expressing cuts corresponding to a vertex $z$ of finite degree not in the image of $\nu$.   The vertex $z$ is included in the expression for $A^*$ if and only if only  finitely many cuts in $\ce _n$ incident with $z$ and pointing at $z$ are subsets of $A$.
     Suppose then that the hypothesis is true for elements $B \in \B _nX$ for which $\mu (B) < \mu (A)$.
Let $C \in \ce _n$ be not nested with $A$.   Then $ \mu (A\cap C) +\mu (A\cap C^*) \leq  \mu (A)$.      Thus each of $A\cap C$ and $A\cap C^*$
can be expressed in a unique way in terms of the $\ce _n$.  If at most one of these expressions involves $C$ then we take the expression for 
$A$ to be the union of the two expressions for $A\cap C$ and $A\cap C^*$.   If both of the expressions involve $C$, then we take the expression
for $A$ to be the union of the two expression with $C$ deleted.      The expression obtained for $A$ is independent of the choice of $C$.
In fact the decomposition will involve precisely those $C$ for which  $C$ occurs in just one of  the decompositions for  $A\cap C$ and $A\cap C^*$.
We therefore have a canonical decomposition for $A$.     To further  clarify this proof observe the following.   The edges $C$ which are not nested
with $A$ form the edge set of a finite subtree $F$ of $T_n$.  If $EF \not= \emptyset$   we can  choose $C$ so that it is a twig of $F$, i.e. so that one vertex $z$  of $F$
is only incident with a single edge $C$ of $F$.  By relabelling $C$ as $C^*$ if necessary we can assume that $\mu (A\cap C) = 0$. 
The vertex determined by $A\cap C$ as above is $z$,  and we have spelled out the recipe for if this vertex is to be included in the expression
for $A$.   The induction hypothesis gives us a canonical expression for $A\cap C^*$, which together with the expression for $A\cap C$
gives the expression for $A$.

\subsection  {Flows in Networks} In this subsection we give a version of the Max-Flow Min-Cut Theorem for arbitrary networks that reduces to the
the usual theorem for a finite network.
Let $X, N$ be as before.
 For $s, t \in VX\cup \Omega X$ an {\it $(s, t)$-flow } in $N$ is a map  $f :  EX \rightarrow \{ 0, 1, 2, \dots \}$  together with an assignment of a direction to each edge  $e$ for which $f(e) \not= 0$  so that its vertices are $\iota e$ and $\tau e$ and the following holds.

\begin {itemize}
\item [(i)]  For each $v \in VX$ there are only finitely many incident edges $e$ for which $f(e) \not= 0$.
\item [(ii)] If  $f^+(v) = \Sigma ( f(e) | \iota e = v)$ and $f^-(v) = \Sigma (f(e) | \tau e =v) $, then $f^+(v) =f^-(v)$ for every $v \not= s, t$.

\item [(ii)]    For every cut $A$ that does not separate $s,t$   If we put $f^+ (A) =  \Sigma  ( f(e) | e\in \delta A,  \iota e \in A) $ and $f^-(A) = \Sigma (f (e) | e\in \delta A,  \iota e \in A^*)$, then  we have $f^+(A) = f^-(A)$. 
That is, for  every cut that does not separate $s, t$ , the flow into the cut is the same as the flow out.

\end {itemize}  
\begin {prop} For any $(s, t)$-flow and any cut $A$ such that $s\in A, t\in A^*$, the value  $f^+(A)  -  f^-(A)$ does not depend on $A$.  This value 
is denoted $|f|$.
\end {prop}
\begin {proof}
Let $A, B$ be cuts separating $s, t$.    Because $A\cap B$ also separates $s, t$, it  suffices to prove that $f^+(A)  -  f^-(A) = f^+(B)  -  f^-(B)$
when $A\subset B$.    Let $e \in \delta A$.  Either $e \in \delta B$ or $e \in \delta (B\cap A^*)$.   If $e' \in \delta B$ is not in $\delta A$ then $e' \in \delta 
(B\cap A^*$ and $\delta (B\cap A^*)$ partitions into those edges with both vertices in $A$ and those with both vertices not in $A$.
Since $A^*\cap B$ does not separate $s,t$,  $f^+(A^*\cap B) = f^-(A^*\cap B)$  and the value of $f^+ -f^-$ on the edges of $\delta (A^*\cap B)$
that are in $A$ is minus the value on the edges not in $A$.    The symmetric difference of $\delta A$ and $\delta B$ consists of the edges
in $\delta (A^*\cap B)$ and it follows that $f^+(A)  -  f^-(A) = f^+(B)  -  f^-(B)$.

\end {proof}
Let $T = T_n$ be the structure tree for $N(X)$.   We have a network $N(T)$ based on $T$ in which each edge has capacity at most $n$.
If $u,v \in VT$ there is a unique maximal $(u,v)$-flow $f(u,v)$ in which $|f(u,v)|$ is the minimal capacity of an edge in the geodesic path $[u,v]$ joining $u, v$.    
An $(s, t)$-flow $f$ in $N$ such that $|f| \leq n$, induces a $(\nu s, \nu t)$-flow  $\bar f$  in $T$  such that  $|f| = |\bar f|$.   In  $T$ an end corresponds to a set of  rays.  For any vertex of $T$ there is a unique ray starting at that vertex and belonging to the end.    If  $u\in VT$ and $v \in  \Omega T$, then there is a  unique ray starting at $u $ and representing $v$.  While if $u,v \in \Omega T$ , then there is a unique two ended path  in $T$ whose ends represent  $u$ and $v$.    In each case we get a $(u,v)$-flow in $T$ by assigning a constant value on the directed edges of the path, provided this constant is less than or equal to the capacity of the each edge in the path.    Every $(u,v)$-flow is of this type.     A maximal $(u,v)$-flow $f(u,v)$ is obtained by taking this constant to be
the minimal capacity of an edge in the path.

\begin {theo} [MFMC]  Let $N$ be a network based on a graph $X$.   Let $s, t \in VX \cup \Omega X$.   The maximum value of an  $(s, t)$-flow is the minimal capacity of a cut separating $s$ and $t$.
\end {theo}
\begin {proof}

Let $n$ be the minimal capacity of a cut in $X$ separating $s, t$.  In the structure tree $T =T_n$  there is a flow from $\nu s$ to $\nu t$ with the property that the value of the flow is $n$.  
     We have to show that each such  flow corresponds to a flow in $X$.   

If $s, t \in VX$, then this follows from the usual proof of the theorem, which we repeat here.
Suppose we have an $(s, t)$-flow $f$ in $N$.   Let $e_1, e_2, \dots e_k$ be a path $p$ joining $s$ and $t$ with the following
property.    Each edge $e_i$ is given an orientation in the flow $f$.   This orientation will not usually be the same as that of going from
$s$ to $t$.    We say that $p$ is an $f$-augmenting path if for each $e_i$ for which $\iota e $ is $s$ or  a vertex of $e_{i-1}$ we have $f(e_i) < c(e_i)$,
and for each edge $e_i$ for which $\iota e = t$ or $\iota e $ is a vertex of $e_{i+1}$ we have $f(e_i) \not= 0$.   
For any flow augmenting path $p$ we get a new flow $f^*$ as follows.   

\begin {itemize} 

\item [(i)]  If $e \in EX$ is not in the path $p$, then $f^*(e) = f(e)$. 
\item [(ii)]  If $e$ is in the path $p$ and $f(e) =0$, then orient 
$e$ so that $\iota e $ is $s$ or  a vertex of $e_{i-1}$, and put  $f^*(e) = 1$.  Recall that we are assuming that $c(e) \not= 0$ for every $e \in EX$ and 
so we have $f^*(e) \leq c(e)$.
\item [(iii)] If $e$ is in the path $p$ and $\iota e $ is $t$ or  a vertex of $e_{i+1}$(so that $\bar e = e_{i+1}$) and  $f(e) \not= 0$ , then $f^*(e) = f(e) -1$.
\item [(iv)] If $e$ is in the path $p$ and $\iota e $ is $s$ or  a vertex of $e_{i-1}$  (so that $e = e_i$), then $f^*(e) = f(e) +1$.

\end{itemize}

The effect of changing $f$ to $f^*$ is to increase the flow along the path $p$.  We have $|f^*| = |f| +1$.

Let $S_f \subset VX$ be the set of vertices  that can be joined to $s$ by a flow augmenting path.  If $t \in S_f$, then we can use the flow augmenting
path joining $s$ and $t$ to get a new flow $f^*$.     We keep repeating this process until we obtain a flow $f$ for which $S_f$ does not contain $t$.
But now $S_f$ is a cut separating $s$ and $t$.   Also if $e \in \delta S_f$, then we have $\iota e \in S_f$ and $f(e) = c(e)$, since otherwise we can
extend the $f$-augmenting path from $s$ to $\iota e$ to an $f$-augmenting path to $\tau e$.    Thus $|f| = c(S_f)$.
But $n$ is the minimal capacity of a cut separating $s$ and $t$ and so $|f| \geq n$.   But also $|f|$ must be less than the capacity of any cut
separating $s$ and $t$ and so $|f| = n$, and  $S_f$ is a minimal cut separating $s$ and $t$.   

If $s \in VX$ and $t \in \Omega X$, then we can build up a flow from $s$ to $t$ in the following way.     Let $D$ be a cut in $ET$ separating $s$ and $t$, so that $s \in D$  and $c(D)\geq n$.  Let $X_D$ be the graph  defined as follows.
The  edge set $EX_D$ consists of all edges $e$ of $X$ which have at least one vertex in $D$, so that either $e \in \delta D$ or $e$ has both vertices
in $D$.   The vertex set $VX_D$ consists of the vertices of $EX_D$, except that we identify all such vertices that are in $D^*$.     Let this vertex be denoted $d^*$.    Thus in $X_D$ the edges incident with $d^*$ are the edges of $\delta D$.      Since $c(D) \geq n$, then as in  the previous case there is a
flow $f_D$ from $s$ to $d^*$ such that $|f_D| = n$.     Let $X_{D^*}$ be the graph defined as for $X_D$, using $D^*$ instead of $D$.   We now have a vertex $d \in VX_{C^*}$ whose incident edges are the edges of $\delta D$.       Now choose another edge $E \not= D$, such that $E^* \subset D^*$, separating $s,t$.
Thus $s \in E$.       Now form a graph $X(D^*, E)$ whose edge set consists of those edges that have at least one vertex in $D^*\cap E$ and whose 
vertex set is the set of vertices of the set of edges except that we identify the vertices that are in $D$ and also identify the vertices that are in $E^*$.
Thus in $X(D^*, E)$ there is a vertex $d$ whose incident edges are those of $\delta D$ and a vertex $e^*$ whose incident edges are those of 
$\delta E$.  The flow $f$ already constructed takes certain values on the edges of $\delta D$.   We can find a $(d, e^*)$ flow which takes these
same values on $\delta D$.    This flow together with the original flow will give an $(s, e^*)$-flow also denoted $f$ such that $|f| = n$.
We can keep on repeating this process and obtain the required $(s, t)$-flow.

 If $s, t$ are both in $\Omega X$,  choose a minimal cut $M$ separating $s$ and $t$, so that $s \in M, t \in M^*$.
Let $X_s$ be the graph  defined as follows.
The  edge set $EX_s$ consists of all edges $e$ of $X$ which have at least one vertex in $M$, so that either $e \in \delta M$ or $e$ has both vertices
in $M$.   The vertex set $VX_s$ consists of the vertices of $EX_s$, except that we identify all such vertices that are in $M^*$.     Let this vertex be denoted $m_t$.    Thus in $X_s$ the edges incident with $m_t$ are the edges of $\delta M$.      If $c(M) =n$, then by the previous case there is a
flow $f_s$ from $s$ to $m_t$ such that $|f_s| = n$.     If we carry out a similar construction for $M^*$ we obtain a flow $f_t$ from $m_s$ to $t$ with
$|f_t| = n$.   We can then piece these flows together to obtain a flow in $X$ from $s$ to $t$ with $|f| = n$.
\end {proof}
The following interesting fact emerges from the above proof in the case when $s, t \in VX$.
If $s,t \in VX$, the cuts in $C \in ET_n$ such that $s \in C, t \in C^*$ form a finite totally ordered set.   It is the geodesic in $ET_n$ joining $\nu s$ and $\nu t$.
Let $D$ be the smallest minimal cut with this property.     Then $S_f \subseteq D$, since for any vertex $u \in D^*$ there can be no $f$-augmenting path
joining $s$ and $u$.    But this must mean that $S_f = D$, since $S_f \in \B _nX$ which is generated by $ET_n$.
Although the maximal flow between $s, t$ is not usually unique, the smallest minimal cut separating $s, t$ is unique.
The way of obtaining $D$ by successively increasing the flow between $s$ and $t$ is obviously not a canonical process, as we choose  flow
augmenting paths  to increase the flow. 

One might think that one could use a structure tree approach to reduce any question about cuts and flows to a one about cuts and flows in a structure tree.
However this is not always possible.   Thus earlier in this section it is shown  that a cut $A$ in $\B _nX$ has a canonical representation in terms of $\ce _n$
and therefore corresponds to a cut $A'$ in $\B _nT$.    However the capacity of $A$ is not usually the same as the capacity of $A'$.   Thus in
Fig \ref {ladder} the cut $A$ corresponding to the brown dashed line has capacity $3$.   This cut is the union of a red cut (with capacity $2$) and a blue cut (with capacity $3$) and in 
$T_3$ corresponds to a cut with capacity $5$.

\section {Almost Invariant Sets}

\subsection {Relative Structure Trees}
We prove Conjecture \ref {KC} in the case when $G$ is finitely generated over $H$, i.e.  $G$ is generated by $H\cup S$ where $S$ is finite.

First, we explain the strategy of the proof.     Suppose that we have a  non-trivial $G$-tree $T$ in which every edge orbit contains an edge  which has an $H$-finite stabiliser, and 
suppose there is a vertex $\bar o$ fixed by $H$.
Let $T_H$ be an   $H$-subtree of  $T$  containing $\bar o$  and every edge with $H$-finite stabiliser.  The action of $H$ on $T_H$ is a trivial action, since 
it has a vertex fixed by $H$, 
and so  the orbit space $H\backslash T_H$ is a tree, which might well be finite, but must have at least one edge.     Our strategy is to show that if
$G$ is finitely generated over $H$ and there is an $H$-almost invariant set $A$ satisfying $AH =A$, then we can find a $G$-tree $T$ with the required properties by first deciding what $H\backslash T_H$ must be and then lifting to get $T_H$ and then $T$.

     We show that if $G$ is finitely generated over $H$, then there is
a $G$-graph $X$ if which there is a vertex with stabiliser $H$ and in which a proper $H$-almost invariant set $A$ satisfying $AH = A$ corresponds to a proper set of vertices with $H$-finite coboundary.     It then follows from Theorem \ref {maintheoremB}, that there is a sequence of structure trees
for $H\backslash X$.      We choose one of these to be $H\backslash T_H$, and show that we can lift this to obtain $T_H$ and then $T$ itself.

For example if $G = {H*}_KL$ then there is a $G$-tree $Y$ with one orbit of edges and  a vertex $\bar o$ fixed by $H$,  and every edge incident with
$\bar o$ has $H$-finite stabiliser.   Suppose that $K, L$ are  such that these are the only edges with $H$-finite stabilisers.   Then $H\backslash T_H$ 
has two vertices and one edge.      When we lift to $T_H$ we obtain an $H$-tree of diameter two in which the middle vertex $\bar o$ has stabiliser
$H$.   The tree $T$ is covered by the translates of $T_H$.   

On the other hand, if $G = L*_KH$ where  $K$ is finite, and $T$ is as above,  then every edge of $T$ is $H$-finite and so $T_H$
is $T$ regarded as an $H$-tree.     The fact that our construction gives a canonical construction for $H\backslash T_H$ means that 
when we lift to $T_H$ and $T$ we will get the unique tree that admits the action of $G$.

We proceed with our proof.
\begin {lemma}  The group $G$ is finitely generated over $H$ if and only if there is a connected $G$-graph $X$ with one orbit of vertices, and finitely
many orbits of edges, and there is a vertex $o$ with stabiliser $H$.
\end {lemma}
\begin {proof}   Suppose $G$ is generated by $H\cup S$, where $S$ is finite.
 Let  $X$ be the graph with $VX = \{ gH | g \in G\}$ and
in which $EX$ is the set of unordered pairs $\{ \{ gH, gsH \},   g \in G, s \in S\} $.    We then have that   $X$ is vertex transitive,  there is a vertex
$o = H$ with stabilizer $H$ and $G\backslash X$ is finite.  We have to show that $X$ is connected.   Let $C$ be the component of $X$ containing 
$o$.     Let $G'$ be the set of those $g \in G$ for which $gH \in C$.   Clearly $G'H = G'$ and $G's = G'$ for every $s \in S$.   Hence $G' = G$ and $C = X$.   Thus $X$ is connected.

Conversely let  $X$ be a connected $G$-graph and $VX = Go$ where $G_o = H$.      Suppose $EX$ has finitely many $G$-orbits,  $Ge_1, Ge_2,
\dots , Ge_r$ where $e_i$  has vertices $o$ and $g_io$.    It is not hard to show that $G$ is generated by $H\cup \{ g_1, g_2, \dots , g_r\} $.

\end{proof}

Let $A \subset G$ be  a proper $H$-almost invariant set satisfying $AH = A$.   Let $G$ be finitely generated over $H$, and let $X$ be a $G$-graph
as in the last lemma.
There is a subset of  $VX$ corresponding to $A$, which is also denoted $A$.      For any $x \in G$,    $A +Ax$ is $H$-finite.   In particular
this is true if $s \in S$.   This means that $\delta A$ is $H$-finite.  Note that neither $A$ nor $A^* = VX - A$ is $H$-finite.
Thus a proper $H$-almost invariant set corresponds to a proper subset of $VX$ such that $\delta A$ is $H$-finite.

 From the previous section (Theorem \ref{cE}) we know that $\B(H\backslash X)$ has a uniquely determined nested set of generators $\ce = \ce (H\backslash X)$.
 For $E \in \ce $, let $\bar E \subset VX$ be the set of all $v \in VX$ such that $Hv \in E$.   Let $C$ be a component of $\bar E$.  
 \begin {lemma}  For $h \in H$,  $hC = C$ or $hC\cap C = \emptyset$.   Also $HC = \bar E$,  $h\delta C \cap \delta C =\delta C$ or $h\delta C \cap \delta C = \emptyset$ and $H\backslash \delta C = \delta E$.   
 \end {lemma}
 \begin {proof} Let $h \in H$.   Then $hC$ is also a component of $\bar E$, since $HC \subseteq E$.
Thus either $hC =C$ or $hC\cap C = \emptyset $.   Let $K$ be the stabilizer of $C$ in $H$.    if $v  \in  C$ then $hv \in  C$ if and only if
$h \in K$. Thus $K\backslash C$ injects into $H\backslash C = E$  and $K \backslash \delta C$ injects into $\delta E$.   But $E$ is connected,
and so the image $HC$ is $E$.  It follows that there is a single $H$-orbit of components.

  \end {proof}
 It follows from the lemma  that it is also   the case that $C^*$ is connected, since any component of $C^*$ must have coboundary 
 that includes an edge from each orbit of $\delta C$.    Let $\bar \ce (H, X)$ be the set of all such $C$, and let $\bar \ce _n(H, X)$  be the subset of $\bar \ce (H, X)$ corresponding to those $C$ for which $\delta C$ lies in at most $n$  $H$-orbits.
 
 \begin {lemma}   the set $\bar \ce (H,X)$ is a nested set.  The set $\bar \ce _n(H, X)$ is the edge set of an $H$-tree.
 \end {lemma}
 \begin {proof}  Let $C, D \in \bar  \ce _n(H,X)$.    Then $HC, HD$ are in the nested set $\ce $.    Suppose $HC \subset HD$,   then $C \subset D$ or
 $C\cap D = \emptyset $.  It follows easily that $\bar \ce (H, X)$ is nested.   It was shown in \cite {[D1]} that a nested set $\ce $ is the directed edge set of
 a tree if and only if it satisfies the finite interval condition, i.e. if $C, D \in \ce $ and $C \subset D$,  then there are only finitely many $E \in \ce $ such that
 $C \subset E \subset D$.
 Thus we have to show that $\bar \ce _n(H, X)$ satisfies the finite interval condition.  If  $C\subset D$ and $C\subseteq E \subseteq D$ where
 $C, E,D \in \bar \ce _n(H,X)$, then $HC \subseteq HE \subseteq HD$.  But $\ce _n(H, X)$ does  satisfy the finite interval condition and 
 $HC = HE$ implies $C = E$.     Now let $C\cap D = \emptyset $ and suppose that $o = H \in C^*\cap D^*$.  There are only finitely
 many $E \in \bar \ce _n$ such that $C \subset E$ and $o \in E^*$ or such that $D \subset E^*$ and $o \in E$.     Each $E \in \bar \ce _n$ 
 such that $C \subset E \subset D^*$ has one of these two properties.
 
 \end {proof}

Let $\bar T = \bar T(H)$ be the tree constructed in the last Lemma.   Let $T  = H\backslash \bar T$.
Note that in the above $\bar T(H)$ is the Bass-Serre   $H$-tree associated with the quotient graph $T(H) = H\backslash \bar T(H)$ and the graph
of groups obtained by associating appropriate labels to the edges and vertices of this quotient graph (which is a tree).    Clearly the action of $H$ on $T(H)$ is a trivial action in that $H$ fixes the vertex $\bar o = \nu o$.   The stabilisers of edges or vertices on a path or ray beginning at $\bar o$ will form a non-increasing sequence of subgroups of $H$.
 
We now adapt the argument of the previous section to show that if $A \subset VX$ is such that $\delta A$ lies in at most $n$  $H$-orbits,  then there is a canonical way
of expressing $A$ in terms of the set $\bar \ce  (H, X)$.    In this case we have to allow unions of infinitely many elements  of the generating set.
Our induction hypothesis is that if $\delta A$ lies in at most $n$ $H$-orbits, then $A$ is canonically  expressed in terms of $\bar \ce _n (H, X)$.
First note that there are only finitely many $H$-orbits of elements of $\bar \ce _n = \bar \ce _n(H,X)$ with which $A$ is not nested.  This is because 
if $C \in \bar \ce _n$ is not nested with $A$ and $F$ is a finite connected subgraph of $H\backslash X$ containing all the edges of $H\delta A$,
then $H\delta C$ must contain an edge of $F$ and there are only finitely many elements of $\ce _n$ with this property.
We now let $\mu (A)$ be the number of $H$-orbits of elements of $\bar \ce _n$ with which $A$ is not nested.
If $\mu (A) = 0$, then $A$ is nested with every $C \in \bar \ce _n$.   This then means that if neither $A$ nor $A^*$ is empty and it is not
already in $\bar\ce _n$, then $A$ determines a vertex $z$  of $\bar T_n$  and either $A$ or $A^*$ is the union (possibly infinite) of edges of
$T_n$ that lie in finitely many $H$-orbits.  
 If $A$ is such a union, then
we use this union for our canonical expression for $A$.   If $A$ is not such a union, then $A^*$ is;   we have $A^* = \bigcup \{ C_{\lambda} | \lambda \in \Lambda \}$, where each $C_{\lambda }$ has $\tau C_{\lambda } = z$ and the edges lie in finitely many $H$-orbits.   We write $A = ( \bigcup \{ C_{\lambda} | \lambda \in \Lambda \})^*  = \bigcap    \{ C_{\lambda}^* | \lambda \in \Lambda \}$.         Note that this gives a canonical way of expressing cuts corresponding to a vertex 
that is not in the image of $\nu $ and whose incident edges lie in finitely many $H$-orbits.        Suppose then that the hypothesis is true for elements $B$ for which $\mu (B) < \mu (A)$.
Let $C \in \bar  \ce _n$ be not nested with $A$.   Then $ \mu (A\cap HC) +\mu (A\cap HC^*) \leq  \mu (A)$.      Thus each of $A\cap HC$ and $A\cap HC^*$
can be expressed in a unique way in terms of the $\ce _n$.   We take the expression for 
$A$ to be the union of the two expressions for $A\cap HC$ and $A\cap HC^*$ except that we include  $hC$  for $h \in H$, only if just one of the two expressions involve $hC$.

If $g \in G$, then $g\bar T(H)$ is a $(gHg^{-1})$-tree.  It is the tree $\bar T(gHg^{-1})$ obtained from the $G$-graph $X$ by using the vertex $go$
instead of $o$.    We now show that there is a $G$-tree $T$ which contains all of the trees $g\bar T(H)$.

We know that the action of the group $G$ on $X$  is vertex transitive and that $X$ has a vertex $o$ fixed by $H$.  Also $G$ is generated by
$H\cup S$ where $S$ is finite.

Clearly there is an isomorphism  $\alpha _g : \bar T(H) \rightarrow \bar T(gHg^{-1})$ in which $D \mapsto gD$.

Suppose now that $\nu o \not= \nu (go)$. 
Let  $A, B $ be $H$-almost invariant sets satisfying $AH = A, BH = B$ and let $g \in G$.  We regard $A,B$ as subsets of $VX$, so that 
$\delta A$ and $\delta B$ are $H$-finite.

Suppose that $o \in gB^*$ and $go\in A^*$. 
 The following Lemma is due to Kropholler \cite {[K90]}, \cite {[K91]}.
 We put $K = gHg^{-1}$.
\begin {lemma}\label {Krop}  In this situation $\delta (A\cap  gB)$is  $(H\cap K)$-finite.
\end {lemma}
\begin {proof}   Let $x \in G$.  We show that the symmetric difference $(A\cap gB)x + (A\cap gB)$ is $(H\cap K)$-finite.
Since $A, B$ are $H$-almost invariant, there are finite sets $E, F$ such that $A+ Ax \subseteq HE$ and $B+Bx\subseteq HF$.
We then have

  $$ 
   (A\cap gB)x + (A\cap gB) = Ax\cap (gBx +gB) + (Ax + A)\cap gB 
= Ax\cap gHF + g(g^{-1}HE\cap B). 
  $$
  
  Now $Ax\cap gHF$ is $K$-finite, but it is also $H$-finite because $gH$ is contained in $A^*$, since $go \in A^*$. A set which is both $H$-finite and $K$-finite is $H\cap K$-finite.  Thus $Ax\cap gHF$ is $(H\cap K)$-finite.  Similarly using the fact that  $g^{-1}o \in B^*$, it follows that $g^{-1}HE\cap B$ is $H\cap (g^{-1}Hg)$-finite, and so $g(g^{-1}HE\cap B)$ is $(H\cap K)$-finite.  Thus $A\cap gB$ is $(H\cap K)$-almost invariant.  
  But this means that $\delta (A\cap gB)$ is $(H\cap K)$-finite.
  \end {proof}  
  What this Lemma says is that if $A, gB$ are not nested then there is a special corner - sometimes called the {\it Kropholler  corner} - which is $(H\cap K)$-almost invariant.    
    
 Notice that in the above situation all of $\delta A, \delta (A\cap gB^*)$ and $\delta  (A\cap gB)$ are $H$-finite.
 If we take the canonical decomposition for $A$, then it can be obtained from the canonical decompositions for $A\cap gB$ and $A\cap gB^*$
 by taking their union and deleting any edge that lies in both.   Also $\delta (gB)$ is $K$-finite and the decomposition for
 $gB$ can be obtained from those for $gB\cap A$ and $gB\cap A^*$.     But the edges in the decomposition for $A\cap gB$ which
 is $(H\cap K)$-almost invariant are the same in both decompositions.

  We will now show that it follows from Lemma \ref {Krop} that the set $G\bar \ce _n$ is a nested $G$-set which satisfies the final interval condition,
  and so it is the edge set of a $G$-tree.
  We have seen that $\bar \ce _n$ is a nested $H$-set where  $\ce _n = H\backslash \bar \ce _n$ is the uniquely determined  nested subset of $\B _n(H\backslash X)$ that generates
  $\B _n (H\backslash X)$ as an abelian group. It is the edge set of a tree $T _n(H\backslash X)$.
  
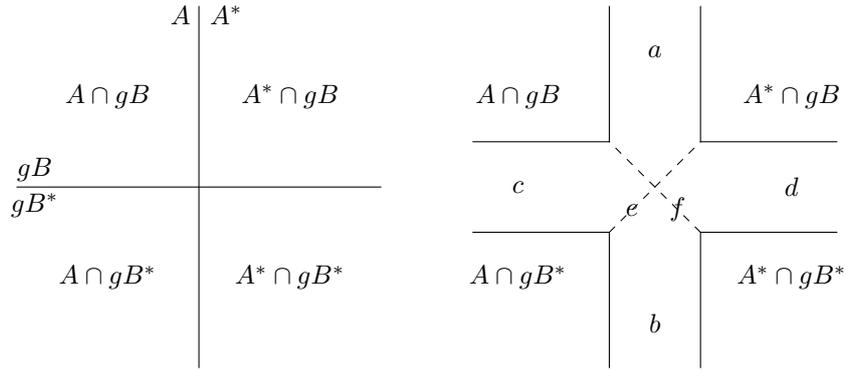
\begin{figure}[ht!]

\centering
\begin{tikzpicture}[scale=.6]

    \draw (0,4)--(8,4);
    \draw (4,0)--(4,8);

\draw (3.6, 7.8) node  {$A$};
\draw (4.6,7.8) node  {$A^*$};
\draw (.4, 4.4) node  {$gB$} ;
\draw (.4,3.6) node  {$gB^*$} ;
\draw (6,6) node  {$A^*\cap gB$};
\draw (2,6) node  {$A\cap gB$};
\draw (2,2) node  {$A\cap gB^*$};
\draw (6,2) node  {$A^*\cap gB^*$};

    \draw (10,5)--(13,5)-- (13, 8);
       \draw (10,3)--(13,3)-- (13, 0);
\draw (18,5)--(15,5)-- (15, 8);
       \draw (18,3)--(15,3)-- (15, 0);

\draw [dashed] (13,3)--(15,5) ;
\draw [dashed] (15,3)--(13,5) ;

\draw (17,6) node  {$A^*\cap gB$};
\draw (11,6) node  {$A\cap gB$};
\draw (11,2) node  {$A\cap gB^*$};
\draw (17,2) node  {$A^*\cap gB^*$};

\draw (14,7) node  {$a$};
\draw (14,1) node  {$b$};
\draw (11, 4) node  {$c$};
\draw (17, 4) node  {$d$};
\draw (13.5, 3.5) node  {$e$};
\draw (14.5, 3.5) node  {$f$};

  \end{tikzpicture}

\vskip .5cm \caption{Crossing cuts}\label{CutsA}
\end{figure}

  If $A, B \in  \bar \ce _n$ and $A, gB$ are not nested for some $g \in G$, then by Lemma \ref {Krop} there is a corner -the Kropholler corner -, which we take to be $A\cap gB$, for which $\delta (A\cap gB)$ is $(H\cap K)$-finite.    We than have canonical decompositions for $A\cap gB$ and $A\cap gB^*$ as above.
  This is illustrated in Fig \ref {CutsA}.     The labels  $a, b, c, d, e, f$ are for sets of edges joining the indicated corners.  In this case the letters do not represent edges of $X$ but elements of $\bar \ce _n$.    Although each $E \in \bar \ce _n$ comes with a natural direction, in the diagram we only count 
  the unoriented edges,   i.e. we count the number of edge pairs $(E, E^*)$.    In the diagram,  $A\cap gB$ is always taken to be the Kropholler corner.
  Thus we have that any pair contributing to $a, f$ or $e$ must be $(H\cap K)$-finite.   Any pair contributing to $e$ or $b$ must be $H$-finite and any
  pair contributing to $e$ or $d$ must be $K$-finite.
  
  We have that $a + e + f +b = 1$ and $c + e + f + d =1$.     Suppose that the Kropholler corner $A\cap B$ is not empty. It is the case that  each of $o$
  and $go$ lies in one of the other three corners.     We know that $o \in gB^*, go \in A^*$.    If $o \in A\cap gB^*$ and $go \in A^*\cap gB$, then 
  $a = c = 1$ and $e =f = b = d =0$ and $A^*\cap gB^* = \emptyset$.      If $o \in A^*\cap gB$ and $go \in A^*\cap gB^*$, then $a = d =1$ and 
  $A\cap gB^* = \emptyset$,  while if both $o$ and $go$ are in $A^*\cap gB^*$,  then either $a =d =1$  and $A\cap gB^* = \emptyset$ or
   $a =c =1$  and $A^*\cap gB = \emptyset$ or  $f =1$ and both
  $A\cap gB^*$ and $A^*\cap gB$ are empty, so that  $A = gB$.  In all cases $A, gB$ are nested.

\begin{figure}[htbp]
\centering
\begin{tikzpicture}[scale=.5]
\draw [red]  (-.8, 1) -- (-.5,2) -- (-.5,3)-- (-1, 4) -- (-1, 5) -- (-1.4, 4) -- (-1. 6, 3)  ;
\draw [thick, dashed, red]   (-2, 2) --(-1.6,3) -- (-1.4, 2) ;
\draw [thick, brown]   (-.5,0) -- ( -.3, 1) --(-.5, 2)-- (1,2)--(2,1)--(3,2)--(4,1)--(5,2)--(6,1)--(7,2)--(8,1)--(9,2)--(10.5,2)  ;
\draw (-1,5) node {$\bullet$} ;
\draw (10.5,5) node {$\bullet$} ;
\draw [thick, blue]  (10.5,2) --(10.5, 3) --(11,4)--(10.5, 5) -- (11.5, 4) --( 11.5, 3)  ;
\draw [thick, brown] (3,-1) --( 2.5, 0) -- (3, 2) -- (3.5, 0) ;
\draw [thick, brown] (4.5,-1) --( 4.5, 0) -- (5, 2) -- (5.5, 0) ;

\draw (-1,5) node [above] {$o$} ;
\draw (10.5,5) node [above] {$go$} ;
\end{tikzpicture}
\end{figure}

 We need also to show that $G\bar \ce _n$ satisfies the finite interval condition.    Let $g \in G$ and let $K = gHg^{-1}$.  Consider the union
 $\bar \ce \cup  g\bar \ce $.     This will be a nested set.    In fact it will be the edge set of a tree that is the union of the trees $T(H)$ and $T(K)$.
 In the diagram the red edges are the edges that are just in $T(H)$.   The blue edges are the ones that are in $T(K)$.  The brown edges are in
 both $T(H)$ and $T(K)$.   An edge is in the geodesic joining $o$ and $go$ if and only if it has stabiliser containing $H\cap K$, it will also
 lie in both $T(H)$ and $T(K)$ (i.e.  it is coloured brown) if and only if it its stabiliser contains $H\cap K$ as a subgroup of finite index.
 It may be the case that $T(H)$ and $T(K)$ have no edges in common, i.e. there are no brown edges.    
 An edge lies in both trees if and only if it has a stabiliser that is $(H\cap K)$-finite.   It there are such edges then they will be the edge set of
 a subtree of both trees. They will correspond to the edge set  $\bar \ce (H\cap K)$.

It  follows that $T(H)$ is always a subtree of a  tree constructed from a subset of $G\bar \ce _n$ that contains $\bar \ce _n$.
If $T(H)$ and $T(K)$ do have an edge in common, then $T(H)\cup T(K)$ will be a subtree of the tree we are constructing.   
If $e \in EX$ has vertices $go$ and $ko$ and there is some $C \in G\bar \ce _n$ that has $e \in \delta C$, then $C \in gET(g^{-1}Hg)\cap kET(k^{-1}Hk)$.   If there is no such $C$,  i.e. there is no cut $C \in G\bar \ce _n$ that separates $o$ and $k^{-1}go$ then $T(H) = k^{-1}gT(H)$.
As there is a finite path connecting any two vertices $u,v$  in $X$, it can be seen that there are only finitely many edges in $G\bar \ce _n$ separating
$u$ and $v$ since any such edge must separate the vertices of one of the edges in the path.   Thus $G\ce _n$ is the edge set of a tree.


  We say that a $G$-tree $T$ is reduced if for every $e \in ET$, with vertices $\iota e$ and $\tau e$  we have that either  $\iota e$  and $\tau e$ are in the same orbit, or $G_e$ is a proper subgroup of both $G _{\iota e} $ and $G _{\tau e} $.
  
   \begin {theo}\label {tree}
   Let $G$ be a group that is finitely generated over a subgroup $H$.
   The following are equivalent:-
    \begin {itemize}
    \item [(i)]
  There is a proper $H$-almost invariant set $A = HAK$ with left stabiliser $H$ and right stabiliser $K$, such that $A$ and $gA$ are nested for every $g \in G$.
  \item [(ii)]  There is a reduced $G$-tree $T$ with vertex $v$ and incident edge $e$ such that $G_v = K$ and $G_e = H$.   
  \end {itemize}
  \end {theo}
  \begin {proof}
  It is shown that (ii) implies (i) in the Introduction.
  
  Suppose than that we have (i).     We will show that there is a $G$-tree - in which $G$ acts on the right - which contains the set 
  $V =  \{ Ax | x \in G\}$ as a subset of the vertex set.        Let  $x \in G$,   then $A + Ax$ is a union of finitely many cosets $Hg_1, Hg_2, \dots , Hg _k$.
  Then $\{ g_1 ^{-1}A, g_2^{-1}A, \dots , g_k^{-1}A  \}$ is the edge set of a finite tree  $F$.
  We know that the set $\{ gA | g \in G\}$ is the edge set of a $G$-tree $T$ provided we can show that it satisfies the finite interval condition.
 But  this must be the case as the edges separating vertices $A$ and $Ax$ will be the edges of  $F$.  
    
  \end {proof}
    \begin {theo} Let $G$ be a  group and let $H$ be a subgroup, and suppose $G$ is finitely generated over $H$.   There is a proper $H$-almost invariant subset $A$
  such that $A = AH$, if and only if  there is a non-trivial  reduced $G$-tree $T$ in which $H$ fixes a vertex and every edge orbit contains an edge  with an $H$-finite edge stabilizer.
  \end {theo}
  
   \begin {proof}
   The only if part of the theorem is proved in Theorem \ref {tree}.   In fact it is shown there that if $G$ has an action on a tree with the specified properties, then there is a proper $H$ almost invariant set $A$  for which $HAH= A$.
   
   Suppose then that $G$ has an $H$-almost invariant set $A$ such that  $AH = A$.  Since $G$ is finitely generated over $H$, we can construct
   the $G$-graph $X$ as above, in which $A$ can be regarded as a set of vertices for which $\delta A$ lies in finitely many $H$-orbits.
   Let this number of orbits be $n$.    Then we have seen that there is a $G$-tree $\bar T_n$ for which $H$ fixes a vertex $\bar o$ and
   every edge is in the same $G$-orbit as an edge in $\bar T(H)$.   The edges in this tree are $H$-finite.  The set $A$ has an expression in terms 
   if the edges of $\bar T(H)$.
   Finally we need to show  that the action on $\bar T_n$ is non-trivial.      If $G$ fixes $\bar o$, then $\nu (A)$ consists of the single vertex
   $o$ and so $A$ is not proper.     In fact the fact that $A$ is proper ensures that no vertex of $\bar T_n$ is fixed by $G$.

   It can be seen from the above that $ \bar T(H) \cap \bar T(g^{-1}Hg)  = \bar T(H\cap gHg^{-1})$  so that  if $e \in ET(H)$, and $g \in G_e$, then $e\in \bar T(gHg^{-1})$ and so  $G_e$ is $H$-finite.
  \end {proof}

    The Kropholler Conjecture follows immediately from the last Theorem.

 \section {$H$-almost stability}
Let $G$ be a group with subgroup $H$, and let $T$ be a $G$-tree.

Let $ \bar  A \subset VT$ be  such that $\delta \bar A \subset ET$  consists of finitely many $H$-orbits of edges $e$ such that $G_e$ is $H$-finite.  Also let $H$ fix a vertex of $T$.    Note that $\delta \bar A$ consists of whole $H$-orbits, so that $e \in \delta \bar A$ implies
$he \in \delta \bar A$ for every $h \in H$.     The fact that $G_e$ is $H$-finite for $e \in \delta \bar A$  follows from the fact that $\delta \bar A$ is $H$-finite.      If $H_e$ is the stabiliser of $e \in \delta \bar A$, then $[G_e : H_e]$ is finite.


Let $v \in VT$,  and   let $A = A(v)  = \{ g \in G | gv \in \bar A\} $.      Note that $A(xv) = A(v)x^{-1}$, so that the left action on $T$ becomes a right 
action on the sets $A(v)$.
 if $x \in G$ and $[v , xv]$ is the geodesic from $v$ to $xv$, then $g \in A + Ax$ if and only the geodesic $[gv, gxv]$ contains an odd number of
 edges in $\delta \bar A$.   If $[v, xv]$ consists of the edges $e_1, e_2, \dots , e_r$,  then $ge_i \in \delta \bar A $ if and only if $Hge_i \in \delta \bar A$.   It follows that $H(A + Ax) = A+Ax$.   It is also clear that for each $e_i$ there are only finitely many cosets $Hg$ such that $Hge_i \in \delta \bar A$.
 Thus $A$ is $H$-almost invariant.     We also have $A (v)H = A(v)$ if $H$ fixes $v$.
 
 For each $e \in ET$, let $d(e)$ be the number of cosets $Hg$ such that $Hge \in \delta \bar A$.    We see that $d(e) = d(xe)$ for every $x \in G$
 and so we have a metric on $VT$, that is invariant under the action of $G$.     We will show that if $G$ has an $H$-almost invariant set such
 that $HAH = A$ then there is a $G$-tree with a metric corresponding to this set.

From now on we are interested in the action of $G$ on the set of $H$-almost invariant sets.   But note that we are interested in
the action by right multiplication.   The  Almost Stability Theorem \cite {[DD]}, also used 
the action by right multiplication. 
Let $A \subset G$ be $H$-almost invariant and let $HA = A$     For the moment we do not assume that $AH = A$.

Let $M  = \{ B | B =_a A\}  $ so that  for  $B, C   \in M ,     B +C  = HF$ where $F$ is finite.

Note  that for $H = \{ 1 \}$  it follows from the Almost Stability Theorem that $M$ is the vertex set of a $G$-tree.

We define a metric on $M$.
For $B, C \in M$ define $d(B, C)$ to be the number of $H$-cosets  in $B +C$.

This is a metric on $M$,   since  $(B+C) + (C+D) = (B+D)$, and so an element  which is in $B+D$ is in just one of 
$B+C$ or $C+D$.  Thus $d(B,D) \leq d(B,C) + d(C, D)$.

Also $G$ acts on $M$ by right multiplication and this action is  by isometries, since $(B+C)z = Bz + Cz$. 
Let $ \Gamma $ be the graph with $V \Gamma = M$ and two vertices are joined by an edge if they are distance one apart.
   Every edge in $ \Gamma $ corresponds to a particular $H$-coset.
There are exactly  $n!$ geodesics joining $B$ and $C$ if $d(B, C) = n$, since a geodesic will correspond to a permutation of
the cosets in $B + C$.  The vertices of $ \Gamma$ on such a geodesic form the vertices of an $n$-cube.

 The edges corresponding to a particular coset $Hb$ disconnect $\Gamma $, since removing this set of edges gives two sets of vertices,   $B$ and $B^*$, where $B$ is the set of those $C \in M$ such that $Hb \subset C$.

It has been pointed out to me by Graham Niblo that $\Gamma $ is related to the  $1$-skeleton of the Sageev cubing introduced in \cite {[Sa]}.
For completeness we describe this connection.

 Let $G$ be a group with subgroup $H$ and let $A = HA$ be an $H$-almost invariant subset.  Let 
 
 $$ \Sigma = \{ gA | g \in G\} \cup \{gA^* | g \in G \}.$$
 
We define a graph $\Gamma '$.
 A vertex $V$ of $\Gamma '$ is a subset of  $\Sigma $ satisfying the following conditions:-
 \begin {itemize}
 \item [(1)] For all $B \in \Gamma '$, exactly one of $B, B^*$ is in $V$.
 \item[(2)] If $B \in V, C \in \Sigma $ and $B \subseteq C$, then $C \in V$.
 \end{itemize}
Two vertices are joined by an edge in $\Gamma '$ if they differ by one element of $\Sigma $.
For $g \in G$, there is a vertex $V_g$ consisting of all the elements of $\Sigma $ that contain $g$.
Then Sageev shows that there is a component $\Gamma ^1$ of $\Gamma '$ that contains all the $V_g$.  

By (1) for each $V \in \Sigma $ either $A \in V$ or $A^* \in V$ but not both.  Let $\Sigma _A$ be the  subset of $\Sigma $ consisting
of those $V \in \Sigma $ for which $A \subset V$.    The edges joining $\Sigma _A$ and $\Sigma _A^*$ in $\Gamma ^1$ form a hyperplane.
Each edge in the hyperplane joins a pair of vertices that differ only on the set $A$.   For each $xA$ there is a hyperplane joining
vertices that differ only on the set $xA$.  Clearly $G$ acts transitively on the set of hyperplanes.

With $V$ as above,  consider the subset $A_V$ of $G$ 
   $$ A_V =\{ x \in G | x^{-1}A \in V \} . $$
   Then $HA_V = A_V$ and $A_{V_1} =A$.
   Also $A_V +A$ is the union of those cosets $Hx$ for which $V$ and $V_1$ differ on $x^{-1}A$, which is finite.   Thus $A_V \in V\Gamma $.

Thus there is a map $V\Gamma ^1 \rightarrow V\Gamma  $ in which $V \mapsto A_V$.
This map is an injective  $G$-map.

If the set $A$ is such that $A$ and $gA$ are nested for every $g \in G$, then $\Gamma ^1$  is a $G$-tree.
Thus $V\Gamma $  contains a $G$-subset that is the vertex set of a $G$-tree.

In $\Gamma $ a hyperplane consists of edges joining those vertices that differ only by a particular coset $Hx$.   Every edge of $\Gamma $
belongs to just one hyperplane.   The group $G$ acts transitively on hyperplanes.   The hyperplane corresponding to $Hx$ has stabilizer
$x^{-1}Hx$.

Suppose now that   $A$ is $H$-almost invariant with $HAK = A$.     Here $H$ is the left stabiliser and $K$ is the right stabiliser of $A$, and we assume that  $H \leq K$, so that in particular $HAH = A$.
Note that it follows from the fact that $A$ is $H$-almost invariant that it is also $K$ almost invariant.
Suppose that $G$ is finitely generated over $K$.     We have seen, in the previous section,  that there is a $G$-tree $T$ in which $A$ uniquely determines a set  $\bar A$ of vertices with
$H$-finite coboundary $\delta \bar A$.  Here $T = T_n$ for $n$ sufficiently large that in the graph $X$ -as defined in the previous section -
the set $\delta \bar A $ is contained in at most $n$ $H$-orbits of edges.
Note that if $e$ is an edge of $\bar T (H) = \bar \ce  (H, X)$, then $\delta e$ is $H_e$-finite, and will consist of finitely many $H_e$-orbits.   It is then 
the case that $[G_e: H_e]$ is finite, since $\delta e$ will consist of finitely many $G_e$-orbits each of which is a union of $[G_e:H_e]$  $H_e$-orbits
of edges.

We also know that $K$ fixes a vertex $\bar o$ of $T$,  and that $H\delta \bar A = \delta \bar A$.  
Thus $\delta \bar A$ consists of finitely many $H$-orbits of edges.   
  We can  contract any edge whose $G$-orbit does not intersect $\delta \bar A$.
We will then have a tree that has the properties indicated at the beginning of this section.   Thus $ \bar  A \subset VT$ is   such that $\delta \bar A \subset ET$  consists of finitely many $H$-orbits of edges $e$ such that  $G_e$ is $H$-finite.
We see that the metric $d$ on $M$   is the same as the metric defined on $VT$.  Explicitly we have proved the following theorem in the case when
$G$ is finitely generated over $K$.
\begin {theo}
Let $G$ be a group with subgroup $H$ and let $A = HAK$ where $H \leq K$ and $A$ is $H$-almost invariant.  
Let $M$ be the $G$-metric space defined above.    Then there is a $G$-tree $T$ such that  $VT$ is a $G$-subset of $M$  and the metric on
$M$ restricts  to a geodesic metric on  $VT$.  If $e \in ET$ then some edge in the $G$-orbit of $e$ has $H$-finite stabiliser.
\end {theo}

This is illustrated in Fig 7 and Fig 8.

\begin {proof}
It remains to show that that the theorem for arbitrary $G$ follows from the case when $G$ is finitely generated over $K$.
Thus if $F$ is a finite subset of $G$,  then there is a finite convex subgraph $C$  of $\Gamma $ containing $AF$.   We can use the graph $X$ of the previous section for the subgroup $L$ of $G$ generated by $H \cup F$ to construct an $L$-tree which has a subtree  $S(F)$ with vertex set contained
in $VC$.  These subtrees have the nice property that if $F_1 \subset F_2$ then $S(F_1)$ is a subtree of $S(F_2)$.    They therefore fit together
nicely to give the required $G$-tree.      We give a more detailed argument for why this is the case.  We follow the approach of \cite {[C]}.

Let $M'$ be the subspace of $M$ consisting of the single $G$-orbit $AG$.    Define an inner product on $M'$  by 
$(B.C)_A = {1\over 2} (d(A, B) +d(A,C) - d(B, C)).$

This turns $M'$ into a $0$-hyperbolic space, i.e. it satisfies the inequality 
$$(B.C)_A \geq min \{ (B.D)_A, (C.D)_A\}$$
for every $B, C, D \in M'$.    This is because we know that if $L \leq G$ is finitely generated over $H$, then there is an $L$-tree
which is a subspace of $M$.   But $A, B, C,D$ are vertices of such a subtree which is $0$-hyperbolic.
It now follows from \cite {[C]},  Chapter 2, Theorem 4.4 that there is a unique $\Z$-tree $VT$ (up to isometry) containing $M'$.    The subset of   $VT$ consisting of vertices of degree larger than  $2$  will be the vertices of a $G$-tree and can be regarded as a $G$-subset of $M$ containing $M'$.
\end {proof}


\begin{figure}[htbp]
\centering
\begin{tikzpicture}[scale=1]
 \path (6,2) coordinate (p1);
    \path (2,3) coordinate (p2);
    \path (5,4) coordinate (p3);
   \draw (4, 3) coordinate (q1);
   \draw (3, 5) coordinate (q3);
    \draw (5.5, 5) coordinate (q4);
    \draw (1.5, 2) coordinate (q5);
    \draw (2.5, 1) coordinate (q6);
     \filldraw (p1) circle (3pt) ;
 \filldraw (p2) circle (3pt);
   \filldraw (p3) circle (3pt);  
\filldraw  (q1) circle (3pt);
\filldraw [white]  (q1) circle (2pt);

\path (3.5, 6) coordinate (p4);
 \filldraw (p4) circle (3pt);  
 \path (3, 2) coordinate (p5);
  \filldraw (p5) circle (3pt);  
 \path (7.5, 2) coordinate (q2);
 \filldraw (q2) circle (3pt);
\filldraw (q3) circle (3pt);
\filldraw (q4) circle (3pt);
\filldraw (q5) circle (3pt);
\filldraw (q6) circle (3pt);

\draw [red] (q3)--(p4) ;
\draw [red] (q4)--(p3) ; 
\draw [red] (p2)--(q5)--(p5);  
 \draw [red] (4,3) --(5.1,3) --(6, 2.7) -- (6,2) ;
\draw [red] (4,3) --(4.9,2.7) --(6, 2.7) ;
\draw [red] (4,3) --(4,2.3) --(4.9, 2) -- (6,2) ;
\draw [red] (4,2.3) --(5.1, 2.3) -- (6,2) ;
\draw [red] (5.1, 2.3) -- (5.1,3) ;
\draw [red] (4.9, 2.7) -- (4.9,2) ;
\draw [red] (4, 3)--(4.5, 3.2) --(5, 4) -- (4.5, 3.8) -- cycle ;
\draw [red] (4, 3)--(3, 3.3) --(2,3) -- (3,2.7) -- cycle ;

 \draw [red] (3.5,6) --(4.6,6) --(5.5, 5.7) -- (5.5,5) ;
\draw [red] (3.5,6) --(4.4,5.7) --(5.5,5.7) ;
\draw [red] (3.5,6) --(3.5,5.3) --(4.4, 5) -- (5.5,5) ;
\draw [red] (3.5,5.3) --(4.6,5.3) -- (5.5,5) ;
\draw [red] (4.6, 5.3) -- (4.6,6) ;
\draw [red] (4.4, 5.7) -- (4.4,5) ;
\draw [red] (6,2)--(6.75, 2.5) --(7.5,2) -- (6.75, 1.5) -- cycle ;
\draw [red] (2.5,1)--(2.5, 1.5) --(3,2) -- (3,1.5) -- cycle ;

    \filldraw [white] (p1) circle (2pt);
\filldraw [white]  (p2) circle (2pt);
\filldraw [white] (p3) circle (2pt);
\filldraw [white] (p5) circle (2pt);
\filldraw [white] (p4) circle (2pt);
\filldraw [white]  (q1) circle (2pt);
 \filldraw [white] (q2) circle (2pt);
\filldraw  [white] (q3) circle (2pt);
\filldraw [white]  (q4) circle (2pt);
\filldraw [white]  (q5) circle (2pt);
\filldraw [white] (q6) circle (2pt); 

\draw (5, 0) node {Fig 7};

\end{tikzpicture}
\end{figure}

\begin{figure}[htbp]
\centering
\begin{tikzpicture}[scale=1.2]
 \path (6,2) coordinate (p1);
    \path (2,3) coordinate (p2);
    \path (5,4) coordinate (p3);
   \draw (4, 3) coordinate (q1);
   \draw (3, 5) coordinate (q3);
    \draw (5.5, 5) coordinate (q4);
    \draw (1.5, 2) coordinate (q5);
    \draw (2.5, 1) coordinate (q6);
     \filldraw (p1) circle (3pt) ;
 \filldraw (p2) circle (3pt);
   \filldraw (p3) circle (3pt);  
\filldraw  (q1) circle (3pt);
\filldraw [white]  (q1) circle (2pt);

\path (3.5, 6) coordinate (p4);
 \filldraw (p4) circle (3pt);  
 \path (3, 2) coordinate (p5);
  \filldraw (p5) circle (3pt);  
 \draw (p1) -- (q1);
  \draw (p2) -- (q1);
\draw (p3) -- (q1);
 \path (7.5, 2) coordinate (q2);
 \filldraw (q2) circle (3pt);
\filldraw (q3) circle (3pt);
\filldraw (q4) circle (3pt);
\filldraw (q5) circle (3pt);
\filldraw (q6) circle (3pt);

\draw (q3)--(p4)--(q4)--(p3);
 \draw (p2)--(q5)--(p5)--(q6);  
 \draw (p1) --(q2);
 \draw [red] (4,3) --(5.1,3) --(6, 2.7) -- (6,2) ;
\draw [red] (4,3) --(4.9,2.7) --(6, 2.7) ;
\draw [red] (4,3) --(4,2.3) --(4.9, 2) -- (6,2) ;
\draw [red] (4,2.3) --(5.1, 2.3) -- (6,2) ;
\draw [red] (5.1, 2.3) -- (5.1,3) ;
\draw [red] (4.9, 2.7) -- (4.9,2) ;
\draw [red] (4, 3)--(4.5, 3.2) --(5, 4) -- (4.5, 3.8) -- cycle ;
\draw [red] (4, 3)--(3, 3.3) --(2,3) -- (3,2.7) -- cycle ;

 \draw [red] (3.5,6) --(4.6,6) --(5.5, 5.7) -- (5.5,5) ;
\draw [red] (3.5,6) --(4.4,5.7) --(5.5,5.7) ;
\draw [red] (3.5,6) --(3.5,5.3) --(4.4, 5) -- (5.5,5) ;
\draw [red] (3.5,5.3) --(4.6,5.3) -- (5.5,5) ;
\draw [red] (4.6, 5.3) -- (4.6,6) ;
\draw [red] (4.4, 5.7) -- (4.4,5) ;
\draw [red] (6,2)--(6.75, 2.5) --(7.5,2) -- (6.75, 1.5) -- cycle ;
\draw [red] (2.5,1)--(2.5, 1.5) --(3,2) -- (3,1.5) -- cycle ;

    \filldraw [white] (p1) circle (2pt);
\filldraw [white]  (p2) circle (2pt);
\filldraw [white] (p3) circle (2pt);
\filldraw [white] (p5) circle (2pt);
\filldraw [white] (p4) circle (2pt);
\filldraw [white]  (q1) circle (2pt);
 \filldraw [white] (q2) circle (2pt);
\filldraw  [white] (q3) circle (2pt);
\filldraw [white]  (q4) circle (2pt);
\filldraw [white]  (q5) circle (2pt);
\filldraw [white] (q6) circle (2pt); 

\end{tikzpicture}
\end{figure}
\eject
 
\begin{figure}[htbp]
\centering
\begin{tikzpicture}[scale=1.1]
 \path (6,2) coordinate (p1);
    \path (2,3) coordinate (p2);
    \path (5,4) coordinate (p3);
   \draw (4, 3) coordinate (q1);
   \draw (3, 5) coordinate (q3);
    \draw (5.5, 5) coordinate (q4);
    \draw (1.5, 2) coordinate (q5);
    \draw (2.5, 1) coordinate (q6);
     \filldraw (p1) circle (3pt) ;
 \filldraw (p2) circle (3pt);
   \filldraw (p3) circle (3pt);  
\filldraw  (q1) circle (3pt);
\filldraw [white]  (q1) circle (2pt);

\draw (3, 5.5) node {$^1$};
\draw (4.6, 5.6) node {$^3$};
\draw (5.1, 4.5) node {$^1$};
\draw (4.4, 3.6) node {$^2$};
\draw (5, 2.6) node {$^3$};
\draw (3, 3.1) node {$^2$};
\draw (2.25,2.1) node {$^1$};
\draw (6.75, 2.1) node {$^2$};
\draw (1.5, 2.5) node {$^1$};
\draw (2.6, 1.5) node {$^2$};

\path (3.5, 6) coordinate (p4);
 \filldraw (p4) circle (3pt);  
 \path (3, 2) coordinate (p5);
  \filldraw (p5) circle (3pt);  
 \draw (p1) -- (q1);
  \draw (p2) -- (q1);
\draw (p3) -- (q1);
 \path (7.5, 2) coordinate (q2);
 \filldraw (q2) circle (3pt);
\filldraw (q3) circle (3pt);
\filldraw (q4) circle (3pt);
\filldraw (q5) circle (3pt);
\filldraw (q6) circle (3pt);

\draw (q3)--(p4)--(q4)--(p3);
 \draw (p2)--(q5)--(p5)--(q6);  
 \draw (p1) --(q2);

    \filldraw [white] (p1) circle (2pt);
\filldraw [white]  (p2) circle (2pt);
\filldraw [white] (p3) circle (2pt);
\filldraw [white] (p5) circle (2pt);
\filldraw [white] (p4) circle (2pt);
\filldraw [white]  (q1) circle (2pt);
 \filldraw [white] (q2) circle (2pt);
\filldraw  [white] (q3) circle (2pt);
\filldraw [white]  (q4) circle (2pt);
\filldraw [white]  (q5) circle (2pt);
\filldraw [white] (q6) circle (2pt); 
\draw (5, 0) node {Fig 8};


\end{tikzpicture}
\end{figure}


\end{document}